\documentclass[11.8pt]{elsarticle}

\usepackage{mathstyle00}
\usepackage{bm}
\usepackage{verbatim}
\usepackage{algorithm}  
\usepackage{algpseudocode}  
\usepackage{amsmath}  

\usepackage[margin=0.92in]{geometry}    
\newtheorem{thm}{Theorem}[section]
\newtheorem{cor}[thm]{Corollary}
\newtheorem{dfn}[thm]{Definition}
\newtheorem{lem}[thm]{Lemma}
\newtheorem{pro}[thm]{Proposition}
\newtheorem{asm}[thm]{Assumption}
\newtheorem{rmk}[thm]{Remark}
\begin{document}
\begin{frontmatter}



\title{A convergent interacting particle method and computation of \\KPP front speeds in chaotic flows}

\author[hku]{Junlong Lyu}
\ead{u3005480@connect.hku.hk}
\author[uoc]{Zhongjian Wang}
\ead{zhongjian@statistics.uchicago.edu}
\author[uci]{Jack Xin}
\ead{jxin@math.uci.edu}
\author[hku]{Zhiwen Zhang\corref{cor1}}
\ead{zhangzw@hku.hk}
\address[hku]{Department of Mathematics, The University of Hong Kong, Pokfulam Road, Hong Kong SAR, China.}
\address[uoc]{Department of Statistics, The University of Chicago, Chicago, IL 60637, USA.}
\address[uci]{Department of Mathematics, University of California at Irvine, Irvine, CA 92697, USA.}
\cortext[cor1]{Corresponding author}

\begin{abstract}
\noindent
In this paper, we study the propagation speeds of reaction-diffusion-advection (RDA) fronts in time-periodic cellular and chaotic flows with Kolmogorov-Petrovsky-Piskunov (KPP) nonlinearity. We first apply the variational principle to reduce the computation of KPP front speeds to a principal eigenvalue problem of a linear advection-diffusion operator with space-time periodic coefficient on a periodic domain. To this end, we develop efficient Lagrangian particle methods to compute the principal eigenvalue through the Feynman-Kac formula. By estimating the convergence rate of Feynman-Kac semigroups and the operator splitting method for approximating the linear advection-diffusion solution operators, we obtain convergence analysis for the proposed numerical method. Finally, we present numerical results to demonstrate the accuracy and efficiency of the proposed method in computing KPP front speeds in   time-periodic cellular and chaotic flows, especially the time-dependent Arnold-Beltrami-Childress (ABC) flow and time-dependent Kolmogorov flow in three-dimensional space. \\
\noindent \textit{\textbf{AMS subject classification:}} 35K57, 47D08, 65C35, 65L20, 65N25.
\end{abstract}
\begin{keyword}
	KPP front speeds;  cellular and chaotic flows; Feynman-Kac semigroups; interacting particle method; 
	eigenvalue problems; convergence analysis.	
\end{keyword}

\end{frontmatter}

\section{Introduction}\label{sec:intro}
\noindent
Front propagation in complex fluid flows arises in many scientific areas such as turbulent combustion, chemical kinetics,  biology, transport in porous media, and industrial deposition processes (see \cite{xin2009} for a review).  A fundamental problem is to analyze and compute large-scale front speeds in complex flows. An extensively studied model problem is the reaction-diffusion-advection (RDA) equation with Kolmogorov-Petrovsky-Piskunov (KPP) nonlinearity \cite{kolmogorov1937}. To be specific, the KPP equation is
\begin{align}
u_t  = \kappa \Delta_{\textbf{x}} u + (\textbf{v} \cdot \nabla_{\textbf{x}}) u +\tau^{-1}f(u), \quad  t\in \mathbb{R}^{+}, \quad \textbf{x}=(x_1,...,x_d)^{T}\in \mathbb{R}^{d},  \label{KPP-eq}
\end{align}
where $\kappa$ is  diffusion constant, $\tau$ is the time scale of reaction rate,   $\textbf{v}$ is an incompressible velocity field (its precise definition will be discussed later),  $u$ is the concentration of reactant or population, and the KPP reaction term  $f(u)=u(1-u)$ satisfying $f(u)\leq uf^{\prime}(0)$. In our analysis and numerical examples, we  will keep $\tau$ and $\kappa$ fixed, while change the magnitude of the velocity field $\textbf{v}$, which equivalently means changing the P\'{e}clet number.

Since the pioneering work of Kolmogorov, Petrovsky, and Piskunov \cite{kolmogorov1937} and Fisher \cite{fisher1937wave} on traveling fronts of the reaction-diffusion equations, this field has gone through enormous growth and development. Reaction-diffusion front propagation in fluid flows has been an active research topic for decades; see e.g.\cite{gartner1979,xin1992existence,majda1994large,xin2000front,berestycki2005speed,nolen2008computing,NolenJack2009,xin2009,nolen2012existence} and references therein. Significant amounts of mathematical analysis and numerical works in this direction have been accomplished when the streamlines of fluid flow are either well-structured (regular motion) or fully random (ergodic motion). Yet, the often encountered less studied case is when the streamlines consist of both regular and irregular motions, while neither one takes up the entire phase space, such as the chaotic Arnold-Beltrami-Childress (ABC) flow \cite{Frisch:86,brummell2001linear} and Kolmogorov flows \cite{KflowGalloway:1992,childress1995stretch}. 

In recent years, much progress has been made in finite element computation of the KPP front propagation in time-periodic cellular and chaotic flows based on a linearized corrector equation.  If the velocity field $\textbf{v}=\textbf{v}(\textbf{x})$ in  the KPP equation \eqref{KPP-eq} is time-independent, the minimal front speed in direction $\textbf{e}$ is given by the variational formula \cite{gartner1979}: $c^{*}(\textbf{e})=\inf_{\lambda>0}\mu(\lambda)/\lambda$, where $\mu(\lambda)$ is the principal eigenvalue of the elliptic operator, $\mathcal{A}^{\lambda}_{1}$, namely,
\begin{align}
\mathcal{A}^{\lambda}_{1} \Phi \equiv \kappa \Delta_{\textbf{x}}\Phi + (2\lambda\textbf{e}+\textbf{v})\cdot
\nabla_{\textbf{x}}\Phi + \big(\kappa \lambda^2+\lambda\textbf{v}\cdot \textbf{e} + \tau^{-1}f^{\prime}(0)\big)\Phi   = \mu(\lambda)\Phi.  \label{elliptic-op}  
\end{align}
In Eq.\eqref{elliptic-op}, $\Phi \in L^2( \mathbb{T}^d)$, $\mathbb{T} = \mathbb{R} / \mathbb{Z}$ is the one-dimensional torus, and $\textbf{v}$ is period $1$ in all direction $x_i , 1 \le i \le d$.
Accurate estimation of $c^{*}(\textbf{e})$ boils down to computing the principal eigenvalue of the operator $\mathcal{A}^{\lambda}_{1}$ in \eqref{elliptic-op}. Adaptive finite element methods (FEM) were successfully applied to solve \eqref{elliptic-op} in \cite{shen2013finite2d,shen2013finite3d}. If the velocity field $\textbf{v}=\textbf{v}(t,\textbf{x})$ in  the KPP equation \eqref{KPP-eq} is periodic in time $t$, then the variational formula  $c^{*}(\textbf{e})=\inf_{\lambda>0}\mu(\lambda)/\lambda$ still holds \cite{nolen2005existence}, where $\mu(\lambda)$ is the principal eigenvalue \cite{hess1991periodic} of the time-periodic parabolic operator, $\mathcal{A}^{\lambda}_{2}$, namely,
\begin{align}
\mathcal{A}^{\lambda}_{2} \Phi \equiv \kappa \Delta_{\textbf{x}}\Phi + (2\lambda\textbf{e}+\textbf{v})\cdot
\nabla_{\textbf{x}}\Phi + \big(\kappa \lambda^2+\lambda\textbf{v}\cdot \textbf{e} + \tau^{-1}f^{\prime}(0)\big)\Phi - \Phi_t = \mu(\lambda)\Phi,  \label{periodic-parabolic-op}  
\end{align}
on the space-time domain $\mathbb{T}^d\times[0,T]$ ($T$ is the period of $\textbf{v}$ in $t$), subject to the same boundary condition in $\textbf{x}$ as \eqref{KPP-eq} and periodic in $t$. An edge-averaged FEM with algebraic multigrid acceleration was developed in \cite{zu2015} to study KPP front speeds in two-dimensional time-periodic cellular flows with chaotic streamlines.  Adaptive FEM methods provide an efficient way to compute the  KPP front speeds in time-periodic cellular and chaotic flows. However, when the magnitude of velocity field is large and/or the dimension of spatial variables is big (e.g. $d=3$), it is extremely expensive to compute  KPP front speeds by using the FEM. 

 


Recently, we have made significant progress in developing Lagrangian particle methods for computing effective diffusivities in chaotic and random flows \cite{WangXinZhang:18,Zhongjian2018sharp,lyu2019convergence}. This motivates us to develop interacting particle methods to compute KPP front propagation in time-periodic cellular and chaotic flows in this paper, especially in three-dimensional flows.

In this paper, we first apply operator splitting methods to approximate the solution operator of the linear advection-diffusion operator (see Eq.\eqref{KPPTimePeriodic}), which is a non-autonomous evolution equation and corresponding to the linearization of the KPP equation. Then, we develop numerical methods to compute the KPP front speeds through the Feynman-Kac formula, which establishes a link between parabolic PDEs and SDEs. Direct approximation of the Feynman-Kac formula is unstable, since the main contribution to the expectation comes from sample paths that visit maximal points of the potential; see Eq.\eqref{Feynman-Kacformula}. Alternatively, we study a normalized version, i.e., the Feynman-Kac semigroup.  Specifically, the principal eigenvalue of $\mathcal{A}_{1}^{\lambda}$ and $\mathcal{A}_{2}^{\lambda}$ can be obtained by studying the convergence of  Feynman-Kac semigroups  for SDEs associated with operators $\mathcal{A}_{1}^{\lambda}$ and $\mathcal{A}_{2}^{\lambda}$ \cite{del2004feynman,ferre2019error}. We approximate the evolution of probability measures by an interacting particle system and use the resampling technique to reduce the variance. Moreover, we estimate the approximation of semigroups associated with the solution operators of non-autonomous evolution equations and obtain convergence analysis for our method in computing the KPP front speeds. 

We point out that using Feynman-Kac semigroups to estimate the principal eigenvalue of differential operators has a long history. It was developed in large deviation theory, where Feynman-Kac semigroups were used to calculate cumulant generating functions \cite{den2008large}. They were also used in important practical applications, such as the diffusion Monte Carlo (DMC) method \cite{foulkes2001quantum}. When the velocity field $\textbf{v}$ of the flow is time-independent, one can apply the backward error analysis approach to obtain the error estimate of the  
principal eigenvalue \cite{ferre2019error}. However, when the velocity field $\textbf{v}$ of the flow is time-dependent, their method cannot be directly applied. There are several novelties in our paper. Firstly, we analyze the solution operator by an operator splitting method and estimate the error in the $L_2$ operator norm. Secondly, we prove the convergence of estimating principal eigenvalues by the Feynman-Kac semigroups for non-autonomous periodic systems. Furthermore, we apply the $N$-interacting particle system ($N$-IPS) method to calculate the principal eigenvalue, where several important 3D chaotic flows are investigated. Notice that when the magnitude of the velocity field is large and/or the dimension of spatial variables is three, it is extremely expensive to calculate the principal eigenvalue using the FEM and spectral method, especially when the flows are time-dependent.

Finally, we carry out numerical experiments to demonstrate the accuracy and efficiency of the proposed method in computing KPP front speeds for time-periodic cellular and chaotic flows. 
Most importantly, we investigate the dependence of  KPP front speeds on the chaos (disorder) and flow intensities. Let $A$ denote the magnitude of the velocity field. For space-time-periodic shear flow, the speed $c^*(A)$ obeys a  quadratic enhancement law: $c^*(A)=c_0(1+\alpha A^2)+ O(A^3)$, $0<A\ll1$, where $c_0$ is the KPP front speed in homogeneous media ($A=0$) and $\alpha>0$ depends only on  flow $\textbf{v}$ \cite{nolenxin2006}. The study for complicated flows, e.g. 3D flows remains largely open. At large $A$, the solution of the principal eigenvalue problem \eqref{elliptic-op} develops internal layers and their locations are unknown \textit{a priori}, which brings difficulties for the FEM and spectral method. We will study this issue in Section \ref{sec:NumCompKPPspead}.  Numerical results show that our interacting particle method is still very efficient when the magnitude of velocity field $A$ is large and computational cost linearly depends on the dimension $d$ of spatial variables in the KPP equation \eqref{KPP-eq}. Thus, we are able to compute the KPP front speeds for time-dependent cellular and chaotic flows of physical interests, including the ABC flows and Kolmogorov flows in three-dimensional space. To the best of our knowledge, our work appears to be the first one in the literature to develop numerical methods to compute KPP front speeds in 3D time-dependent flows. Furthermore, we numerically verify that the relationship between the KPP front speed $c^*(A)$ and the effective diffusivity $D^{E}(A)$, i.e. $c^*(A)=O(\sqrt{D^{E}(A)})$, is true in 2D steady cellular flows and still exists in the 3D Kolmogorov flows. We also compute the invariant measure of Feynman-Kac semigroups by our interacting particle method.


The rest of the paper is organized as follows. In Section 2, we propose Lagrangian interactive particle methods in computing KPP front speeds in time-periodic cellular and chaotic flows. In Section 3, we estimate the approximation of semigroups associated with the solution operators of non-autonomous evolution equations and obtain convergence analysis for our method.  In Section 4, we present numerical results to demonstrate the accuracy and efficiency of our method.
In addition, we investigate the dependence of  KPP front speeds on the chaos (disorder) and flow intensities, especially in 3D time-dependent chaotic flows. Concluding remarks are made in Section 5. Finally, we collect several fundamental results for abstract linear evolution equations by semigroup theory in the Appendix. 
 


\section{Efficient Lagrangian methods in computing KPP front speeds}\label{sec:LagrangianMethods}

\subsection{Computing principal eigenvalue via the Feynman-Kac formula}\label{sec:FKformula}
\noindent
In this section, we develop Lagrangian interacting particle methods to compute KPP front speeds via the Feynman-Kac formula. We consider the linearized corrector equation of the KPP equation \eqref{KPP-eq}, where the velocity field $\textbf{v}(t,\textbf{x})$ is space-time periodic, mean zero, and divergence-free. To compute the KPP front speed  $c^{*}(\textbf{e})$ along direction $\textbf{e}$, let $w$ solve a linearized equation parameterized by $\lambda>0$: 
\begin{equation} \label{KPPTimePeriodic}
w_t = \mathcal{A}w:=\kappa\Delta w + (2\lambda\textbf{e}+\textbf{v})\cdot
\nabla_{\textbf{x}}w + \big(\kappa \lambda^2+\lambda\textbf{v}\cdot \textbf{e} + \tau^{-1}f^{\prime}(0)\big)w, 
\end{equation} 
with initial condition $w(\textbf{x},0)=1$. Then, the principal eigenvalue $\mu(\lambda)$ is given by
\begin{equation} \label{TimePeriodicPrinEig}
\mu(\lambda)=\lim_{t\to \infty}\frac{1}{t}\ln \int_{ \mathbb{T}^d} w(t,\textbf{x})d\textbf{x}.
\end{equation} 
The number $\mu(\lambda)$ is also the principal Lyapunov exponent of the parabolic equation 
\eqref{KPPTimePeriodic}, which is convex and superlinear for large $\lambda$ \cite{nolen2005existence,zu2015}. Finally, we compute the KPP front speed using the variational formula  $c^{*}(\textbf{e})=\inf_{\lambda>0}\mu(\lambda)/\lambda$.
 
To design Lagrangian particle methods, we decompose the operator $\mathcal{A}$ in \eqref{KPPTimePeriodic} into $\mathcal{A}=\mathcal{L}+\mathcal{C}$, where $\mathcal{L}:=\kappa\Delta   + (2\lambda\textbf{e}+\textbf{v})\cdot\nabla_{\textbf{x}}$ and $\mathcal{C}:=c(t,\textbf{x})=\big(\kappa \lambda^2+\lambda\textbf{v}\cdot \textbf{e} + \tau^{-1}f^{\prime}(0)\big)$. To approximate the operator $\mathcal{L}$, we define a SDE system as follows
\textcolor{black}{\begin{equation} 
		{d\textbf{X}^{s,\textbf{x}}_t = \textbf{b}(t,\textbf{X}^{s,\textbf{x}}_t)dt 
			+ \sqrt{2\kappa}d\textbf{w}(t),  \quad \textbf{X}^{s,\textbf{x}}_s=\textbf{x}}, \quad t \ge s, \label{SDE-system}  
\end{equation}}
where the drift term $\textbf{b}=2\lambda\textbf{e}+\textbf{v}$ is determined by the advection field in the operator $\mathcal{L}$  and $\textbf{w}(t)$ is a $d$-dimensional Brownian motion. The principal eigenvalue $\mu(\lambda)$ of \eqref{KPPTimePeriodic} can be represented via the Feynman-Kac formula as follows:
\textcolor{black}{\begin{equation} 
	\mu(\lambda)=\lim_{t\to\infty}\frac{1}{t}\ln\mathbb{E}\Big(\exp\big(\int_{0}^t
	c(t-s,\textbf{X}^{0,\textbf{x}}_s)ds\big)\Big),
	\label{Feynman-Kacformula}  
\end{equation}}
where the expectation $\mathbb{E}(\cdot)$ is over randomness induced by the Brownian motion $\textbf{w}(t)$. 

If we apply the formula \eqref{TimePeriodicPrinEig} to compute the principal eigenvalue $\mu(\lambda)$, we need to solve a parabolic-type PDE \eqref{KPPTimePeriodic} using numerical methods, such as FEM and spectral method. When the magnitude of the velocity field is large and/or the dimension of spatial variables $d$ is big (say $d=3$), the FEM and spectral method become extremely expensive. The Feynman-Kac formula \eqref{Feynman-Kacformula} provides an alternative strategy to design Lagrangian methods to compute the principal eigenvalue $\mu(\lambda)$, and thus allows us to 
compute the KPP front speeds. As we will demonstrate in Section \ref{sec:NumericalResults}, the proposed Lagrangian method is efficient for computing KPP front speeds in 3D time-dependent chaotic flows. 

\begin{rmk}\label{FKformula4timeindflow}
When the velocity field in the KPP equation \eqref{KPP-eq} is time-independent, the construction of the Lagrangian method for computing KPP front speeds is straightforward. We simply replace the drift term $\textbf{b}$ in \eqref{SDE-system} and the potential $c$ in \eqref{Feynman-Kacformula} by their time-independent counterparts.
\end{rmk} 

\subsection{Feynman-Kac semigroups}\label{sec:FKsemigroup}
\noindent
Directly using the Feynman-Kac formula \eqref{Feynman-Kacformula} and Monte Carlo method to compute the principal eigenvalue $\mu(\lambda)$ is unstable as the main contribution to \textcolor{black}{$\mathbb{E}\Big(\exp\big(\int_{0}^t
c(t-s,\textbf{X}^{0,\textbf{x}}_s)ds\big)\Big)$} comes from sample paths that visit maximal or minimal points of the potential function $c$, which leads to inaccurate or even divergent results.
 
Accurate principal eigenvalue $\mu(\lambda)$ can be obtained by studying the convergence of the Feynman-Kac semigroup associated with the SDE system \eqref{SDE-system} and the potential $c$. Specifically, let $\mathcal{P}(\mathbb{T}^d)$ denote the set of probability measures over $\mathbb{T}^d$ and $S = \mathcal{C}^{\infty}(\mathbb{T}^d)$. We define the evolution operator associated with the process $(\textbf{X}^{s,\textbf{x}}_t)_{t \ge s}$ in \eqref{SDE-system} as   
\textcolor{black}{\begin{equation}
		(\nu)(P_{t_2,t_1}\phi) = \mathbb{E}_{{\bf x} \sim \nu}\big(\phi(\textbf{X}^{t_1,\textbf{x}}_{t_2})\big),\quad\forall\nu \in \mathcal{P}(\mathbb{T}^d), ~\phi \in S, ~ t_2 \ge t_1.
		\label{operatorPt}
\end{equation}}
Similarly, we define its weighted counterpart as  
\textcolor{black}{\begin{equation}
		(\nu)(P^c_{t_2,t_1}\phi) = \mathbb{E}_{{\bf x} \sim \nu}\Big(\phi(\textbf{X}^{t_1,\textbf{x}}_{t_2})\exp\big(\int_{t_1}^{t_2}c(t_2 - s,\textbf{X}^{t_1,\textbf{x}}_s)\big)ds\Big),\quad\forall\nu \in \mathcal{P}(\mathbb{T}^d), ~\phi \in S, ~t_2 \ge t_1.
		\label{operatorPct}
\end{equation} }
\textcolor{black}{In other words, the infinitesimal generators of $P_{t_2,t_1}$ and $P^c_{t_2,t_1}$ respect to $t_2$ are $\mathcal{L}(t_1)$ and $\mathcal{A}(t_1)=\mathcal{L}(t_1)+\mathcal{C}(t_1)$, respectively}. \textcolor{black}{Equipped with the definitions of the evolution operators $P_{t_2,t_1}$ and $P^c_{t_2,t_1}$, we can define the Feynman-Kac operator $\Phi^{c}_{t_2,t_1}$ as follows
	\begin{equation} 
		\Phi^{c}_{t_2,t_1}(\nu)(\phi):=\frac{(\nu)(P^c_{t_2,t_1}\phi)}{(\nu)(P^c_{t_2,t_1} 1)}=\frac{\mathbb{E}_{{\bf x} \sim \nu}\big(\phi(\textbf{X}^{t_1,\textbf{x}}_{t_2})\exp\big(\int_{t_1}^{t_2}c(t_2 - s,\textbf{X}^{t_1,\textbf{x}}_s)\big)ds\big)}{\mathbb{E}_{{\bf x} \sim \nu}\big(\phi(\textbf{X}^{t_1,\textbf{x}}_{t_2})\big)}. \label{FeynmanKacsemigroup}  
	\end{equation}
	One can easily verify that for all $\nu \in \mathcal{P}(\mathbb{T}^d)$ and $t_1\le t_2 \le t_3\in \mathbb{R}_{+}$,  $\Phi^{c}_{t_2,t_1}(\Phi^{c}_{t_3,t_2}(\nu))=\Phi^{c}_{t_3,t_1}(\nu)$. 
	Notice that we use $T$ to denote the period of the velocity in time. For convenience we denote $\Phi^{c}_{T} = \Phi^{c}_{T,0}$ and $P_T^c = P_{T,0}^c$. Therefore, we consider the Feynman-Kac semigroup for $t = nT, n \in \mathbb{N}$. Namely, we consider $\Phi^{c}_{nT} = (\Phi^{c}_{T})^n $.  One can 
	easily verify the Feynman-Kac semigroup $\Phi^{c}_{nT}$ satisfies the following property, where the proof is a direct conclusion of Theorem \ref{bam} and Theorem \ref{thm:existenceinvariantmeasure}.}
\begin{pro}\label{convergenceratesemigroup}
For any $\nu \in \mathcal{P}(\mathbb{T}^d)$ and $\phi \in S$, 
there exists $C>0$ such that 
\begin{equation}
\Big|\Phi^{c}_{nT}(\nu)(\phi)-\int_{\Omega}\phi d\nu_{c}\Big|\leq C||\phi||\exp(-\delta_{c}nT), 
\end{equation}  
where $\delta_{c}=\inf\{\mu(\lambda)-\Re(z):z\in \sigma(\mathcal{A}) \setminus\{\mu(\lambda)\}\}>0$ is the spectral gap of the operator $\mathcal{A}$.  
\end{pro}
The exponential-decay property stated above ensures us to obtain an invariant measure $\nu_c$ for $\Phi^{c}_{T}$ from any initial measure $\nu$. From the definition of $\nu_c$, we know that $\Phi_T^c (\nu_c) = \nu_c$, which means that for any $\phi \in S$
\begin{equation}
	\int_{\mathbb{T}^d} \phi d\nu_c = (\int_{\mathbb{T}^d} P_T^c 1 d\nu_c)^{-1} \int_{\mathbb{T}^d} P_T^c \phi d\nu_c.
\end{equation} 
Therefore, we can find that the principal eigenvalue of $P_T^c$ is just $\int_{\mathbb{T}^d} P_T^c 1 d\nu_c$, which provides a feasible way to compute the  principal eigenvalue. 

\subsection{Numerical discretization and resampling techniques }\label{sec:NumApprResampling}
\noindent
Let $M$ be the number of time discretization interval for each period and $\Delta t=T/M$. 
We use the Euler-Maruyama scheme to discretize the SDE \eqref{SDE-system} and obtain
\begin{equation} \label{NumericalMethod4SDE}
{\bf X}_{i+1} =   {\bf X}_{i} + 
\textbf{b}(t_i,{\bf X}_{i})\Delta t+  \sqrt{2\kappa \Delta t} \bm\omega_i,  
\end{equation}
where $t_i = i\Delta t$ and $\bm\omega_i$'s are i.i.d. $d$-dimensional standard Gaussian random variables. The numerical scheme \eqref{NumericalMethod4SDE} defines an evolution operator $P^{\Delta t}_i$ (also known as transition operator) as follows  
\begin{equation} \label{TransitionOperator}
 P^{\Delta t}_i\phi(\bm x)=\mathbb{E}\big(\phi({\bf X}_{i+1})|{\bf X}_{i}=\bm x\big), \quad\phi \in S.
\end{equation} 
The evolution operator $P^{\Delta t}_i$ describes how the values of a given function evolve in $L_2$ sense over one time step $\Delta t$. One can easily verify that  
\begin{equation} \label{MarkovProcessApproximation}
\big|\big|P^{\Delta t}_{i} - e^{\Delta t \mathcal{L}(t_i)}\big|\big|_{L^2} \le C (\Delta t)^2,
\end{equation}
where $C$ is a positive constant \cite{MilsteinTransition}. Specially, when $\textbf{b} = 0$,  $P^{\Delta t}_{t_i} = e^{\Delta t \mathcal{L}(t_i)}$ for all $i$. Therefore,  solving 
the SDE system \eqref{SDE-system} by the numerical scheme \eqref{NumericalMethod4SDE} provides a good approximation to the evolution operator $e^{\Delta t \mathcal{L}(t_i)}$, which plays an important role in the error estiamte of our Lagrangian methods in Section \ref{sec:ConvergenceAnalysis}. 

In addition, we can define the approximation operator for $P^c_t$ in \eqref{operatorPct}. For instance, if we choose the left-point rectangular rule, we obtain  that for any $\nu \in \mathcal{P}(\mathbb{T}^d)$ and $\phi \in S$
\begin{equation}
(\nu)(P_i^{\Delta t} e^{\Delta t \mathcal{C}(t_i)}\phi) = \mathbb{E}\Big(\phi({\bf X}_{i+1})\exp\big(
c(t_{i},{\bf X}_{i+1})\Delta t\big)\big|{\bf X}_{i} \sim \nu \Big), \quad i=0,1,...,M-1.
\label{DiscreteOperatorPct}
\end{equation}


The time discretization for Feynman-Kac semigroup \eqref{FeynmanKacsemigroup} reads:
\begin{equation}
\Phi_{i}^{\mathcal{C},\Delta t} (\nu)(\phi) = \frac{ (\nu)(P_i^{\Delta t} e^{\Delta t \mathcal{C}(t_i)}\phi)}{(\nu)(P_i^{\Delta t} e^{\Delta t \mathcal{C}(t_i)}1)},
\quad i=0,1,...,M-1.
\label{DiscreteFeynmanKacsemigroup} 
\end{equation}
It is difficult to obtain a closed-form solution to the evolution of probability measure in \eqref{DiscreteFeynmanKacsemigroup}. Therefore, we 
approximate the evolution of probability measure in \eqref{DiscreteFeynmanKacsemigroup} by an $N$-interacting particle system ($N$-IPS) \cite{del2000branching}. 
Let us introduce the notation $\mathcal{K}^{\Delta t} = \mathcal{K}^{\Delta t,M-1} \mathcal{K}^{\Delta t,M-2}\cdots \mathcal{K}^{\Delta t,0}$, where $\mathcal{K}^{\Delta t,i} = P_i^{\Delta t} e^{\Delta t \mathcal{C}(t_i)}$, $\Delta t=T/M$, and $T$ is time period. We denote  
\begin{align}
\Phi^{\mathcal{K}^{\Delta t,i}} (\nu)(\phi) = \frac{ (\nu)(\mathcal{K}^{\Delta t,i}\phi)}{(\nu)(\mathcal{K}^{\Delta t,i}1)},
\quad i=0,1,...,M-1, 
\label{operatorPhiKiDeltat}
\end{align}
the Feynman-Kac semigroup associated with the operator $\mathcal{K}^{\Delta t,i}$. Then, according to Lemma \ref{multible}, it satisfies
\textcolor{black}{
	\begin{equation}
		\Phi^{\mathcal{K}^{\Delta t}} = \prod_{i=0}^{M-1} \Phi^{\mathcal{K}^{\Delta t,M-1-i}}=\Phi^{\mathcal{K}^{\Delta t,0}}\Phi^{\mathcal{K}^{\Delta t,1}}\cdots\Phi^{\mathcal{K}^{\Delta t,M-1}}.
		\label{operatorPhiKDeltat}
\end{equation}}

Suppose the Markov process $({\Theta}, (\mathcal{F}_n)_{n\ge 0}, (\bm\xi^n)_{n\ge 0}, \mathbb{P})$ 
is defined in the product space $(\mathbb{T}^d)^N$. 
For any initial probability measure $\pi_0 = \nu$, we approximate it by an $N$-particle system as
\begin{equation}
	P(\bm\xi^0 \in d\bm z) = \prod_{p=1}^N \pi_0(dz^p).
	\label{initialNparticle}
\end{equation}
Then, we evolve the $N$-particle system according to 
\textcolor{black}{
	\begin{equation}\label{evolveNparticle}
		P(\bm\xi^n \in d\bm z|\bm\xi^{n-1} = \bm x) = \prod_{p=1}^N \Phi^{\mathcal{K}^{\Delta t}}(\frac{1}{N} \sum_{i=1}^{N}\delta_{x^i})(dz^p)=\prod_{p=1}^N (\prod_{i=0}^{M-1} \Phi^{\mathcal{K}^{\Delta t,M-1-i}})(\frac{1}{N} \sum_{i=1}^{N}\delta_{x^i})(dz^p),
	\end{equation}}
where $\bm x=(x^1,...,x^N)^T$ and $n$ denotes the iteration number in the evolution of probability measure by the Feynman-Kac semigroup \eqref{DiscreteFeynmanKacsemigroup}. 

Using Eq.\eqref{evolveNparticle}, we can compute the evolution of the $N$-particle system from $\bm\xi^{n-1}$ to $ \bm\xi^{n}$. It will be divided into $M$ small steps. Let us denote $\bm\xi^{n}_0 = \bm\xi^{n}$ for all $n$.  Within each iteration stage, we evolve the particles from $t=0$ to $t=T$ by the evolution operator $P^{\Delta t}_{i}$ and resample these particles according to weights determined by the potential function. Specifically, at $t_i=i\Delta t$, $i=0,...,M-1$, we evolve the particles in  $\bm\xi_{i}^{n-1}=(\xi_{i}^{1,n-1},...,\xi_{i}^{N,n-1})$ by the numerical scheme \eqref{NumericalMethod4SDE} and get $\textcolor{black}{\widetilde{\bm\xi}_{i}^{n-1}}=(\widetilde{\xi}_{i}^{1,n-1},...,\widetilde{\xi}_{i}^{N,n-1})$. Namely, each particle is updated by 
\textcolor{black}{\begin{equation}
	\widetilde{\xi}_{i}^{p,n-1} = \xi_{i}^{p,n-1} + \textbf{b}(t_{M-1-i},\xi_{i}^{p,n-1})\Delta t+  \sqrt{2\kappa \Delta t} \bm\omega_i^{p,n-1}, \quad p = 1,2,...,N, \label{NMSDEevolveparticles}
\end{equation} }
where $ \bm\omega_i^{p,n-1}$'s are i.i.d. $d$-dimensional standard Gaussian random variables.

Then, we resample the particles in $\widetilde{\bm\xi}_{i}^{n-1}$ according to the multinomial distribution with the weights 
\begin{equation}
	{w}_{i}^{p,n-1}=\frac{\exp{\big(c(t_{M-1-i}, \widetilde\xi_{i}^{p,n-1})\Delta t\big)}}{\sum_{p=1}^{N}\exp{\big(c(t_{M-1-i}, \widetilde\xi_{i}^{p,n-1})\Delta t\big)}}, \quad  p=1,2,...,N, 
	\label{MultinomialResampleWeight}
\end{equation} 	 
and obtain $\bm\xi_{i+1}^{n-1}$.
The evolution of $N$-IPS from $(n-1)T$ to $nT$ can be represented as follows
\begin{align}
	&\bm \xi_0^{n-1} = (\xi_{0}^{1,n-1},\cdots,\xi_{0}^{N,n-1})\longrightarrow 
	\bm\xi_1^{n-1} = (\xi_{1}^{1,n-1},\cdots,\xi_{1}^{N,n-1})\longrightarrow \nonumber\\
	&\cdots\longrightarrow 
	\bm\xi_{M}^{n-1} = (\xi_{M}^{1,n-1},\cdots,\xi_{M}^{N,n-1})=
	\bm\xi_{0}^{n} = (\xi_{0}^{1,n},\cdots,\xi_{0}^{N,n}).
\end{align}  
After obtaining the empirical distribution of the particles $\bm\xi_{0}^{n}$, we can compute the principal eigenvalue. At the iteration stage $n$, we first define the change of the mass as follows
\begin{equation}
	e_{i,n}^N = N^{-1}\sum_{p=1}^{N} \exp(c(t_{M-1-i},\widetilde\xi_{i}^{p,n})\Delta t).
\end{equation}
Then, we compute the approximation of the principal eigenvalue by 
\begin{equation}
	\mu^n_{\Delta t}(\lambda) = (M\Delta t)^{-1}\sum_{i=0}^{M-1} \log\Big( N^{-1}\sum_{p=1}^{N} \exp(c(t_{M-1-i},\widetilde\xi_{i}^{p,n})\Delta t)\Big). 
	\label{NIPS-PrincipalEig}
\end{equation} 
We know that the empirical distribution of the particles $\bm\xi_{0}^{n}$ will weakly converge to the distribution $\Phi_n^{\mathcal{K}^{\Delta t}} (\pi_0)$ as $N \to \infty$. Therefore, we can use $\mu^n_{\Delta t}(\lambda)$ to approximate
the principal eigenvalue $\mu(\lambda)$.

Finally, we give the complete algorithm in Algorithm \ref{Algorithm-TimeFependent}. The performance of our method will be demonstrated in Section \ref{sec:NumericalResults}.

\begin{algorithm}[h]   
	\caption{ Algorithm for computing the principal eigenvalues of parabolic equations}  
	\label{Algorithm-TimeFependent}  
	\begin{algorithmic}[1]  
		\Require  
		velocity field $\textbf{v}(\textbf{x},t)$, potential $c(\textbf{x},t)$, number of $N$-IPS system (i.e., $N$), initial probability measure $\nu_0$, iteration number $n$,  time period $T$,
		time step $\Delta t = T/M$ and $t_i = i\Delta t, 0 \le i \le M$. 
		\State Generate $N$ i.i.d. $\nu_0$-distributed random variables on $[0,1]^d$: \textcolor{black}{$\bm \xi_0^{0} =  (\xi_{0}^{1,0},\cdots,\xi_{0}^{N,0})$}, the $N$-particle system.
		\For {$k = 1:n$}
		\For {$i = 0:M-1$}
		\State Generate i.i.d. standard Gaussian random variables $(\omega_i^{1,k-1},...,\omega_i^{N,k-1})$ and compute \textcolor{black}{$\widetilde {\bm\xi}_{i}^{k-1}  =(\widetilde{\xi}_{i}^{1,k-1},...,\widetilde{\xi}_{i}^{N,k-1})$} according to $ {\bm\xi}_{i}^{k-1}$ by \eqref{NMSDEevolveparticles}. 
		\State Compute the pointwise value $\bm S = (e^{C^1},\cdots, e^{C^N})$, where 
		\textcolor{black}{$C^p=c(t_{M-1-i},\widetilde{\xi}_{i}^{p,k-1})\Delta t$}. 
		\State Compute weights $\bm w = (w^1,\cdots, w^N) = \bm S/\text{sum}(\bm S)$
		and ${\bf E}_{k,i} = \frac{1}{\Delta t}\text{log}(\text{mean}(\bm S))$.
		\State Resample \textcolor{black}{$\widetilde{\bm\xi}_{i}^{k-1}$} according to multinomial distribution with weight $\bm w$ \eqref{MultinomialResampleWeight}, and get $\bm\xi_{i+1}^{k-1}$.
		\EndFor 
		\State Compute $\mu^k_{\Delta t}(\lambda)  = M^{-1}\sum_{i=0}^{M-1}({\bf E}_{k,i})$ and define  $\bm\xi_{0}^{k} = \bm\xi_{M}^{k-1}$.
		\EndFor 
		\Ensure   The approximate invariant distribution $\Phi_n^{\mathcal{K}^{\Delta t}}(\nu)$-distributed 
		$N$-particle system $\bm\xi^n_0$ and approximate the principal eigenvalue $\mu^n_{\Delta t}(\lambda)$ using \eqref{NIPS-PrincipalEig}.
	\end{algorithmic}  
\end{algorithm}   
\begin{remark}
	When the flow is time-independent, we can view it as a periodic flow with any given period $T$. Then, we can still use Algorithm \ref{Algorithm-TimeFependent} to compute the principal eigenvalue. Hence the numerical schemes and the convergence analysis proposed in time-dependent flow can be applied by assigning $T=\Delta t$ and $M=1$.
\end{remark}

\section{Convergence analysis of the Lagrangian particle method}\label{sec:ConvergenceAnalysis}
\noindent
In this section, we will prove the convergence of the Lagrangian particle method in computing 
KPP front speed. We divide the analysis into two parts. The first part studies the approximation of the evolution of parabolic operators by using an operator splitting method. The second part studies the error estimate of the Lagrangian particle method in computing the principal eigenvalue of parabolic operators.

\subsection{Approximation of the evolution of parabolic operators}\label{sec:ErrorEstimateOperatorSplit}
\noindent
We first rewrite the linearized corrector equation of the KPP equation \eqref{KPPTimePeriodic} into the following non-autonomous parabolic equation
 \begin{equation}
w_t = \kappa\Delta_{\textbf{x}}  w + \textbf{b}(t,\textbf{x})\cdot\nabla_{\textbf{x}} w + c(t,\textbf{x})w, \quad \textbf{x}=(x_1,...,x_d)^{T}\in \mathbb{T}^d= [0,1]^d, \quad t \in[0,T], 
\label{non-autonomousParabolicEQ}
\end{equation}
where the initial condition $w(0,\textbf{x})=w_0$, $\textbf{b}(t,\textbf{x})=2\lambda\textbf{e}+\textbf{v}$, $c(t,\textbf{x})=\kappa \lambda^2+\lambda\textbf{v}\cdot \textbf{e} + \tau^{-1}f^{\prime}(0)$, and $T$ is final computational time. Since the velocity  $\textbf{v}=\textbf{v}(t,\textbf{x})$ is space-time periodic, so do $\textbf{b}(t,\textbf{x})$ and $c(t,\textbf{x})$. We assume the period of $\textbf{b}(t,\textbf{x})$ and $c(t,\textbf{x})$ is one in each dimension and they are smooth functions. For notational simplicity, we define 
\begin{equation} 
\mathcal{A}(t)=\mathcal{L}(t)+\mathcal{C}(t), 
\label{DefinitionOperatorA}
\end{equation}
where $\mathcal{L}(t):=\kappa\Delta_{\textbf{x}}  + \textbf{b}(t,\textbf{x})\cdot\nabla_{\textbf{x}}$ and $\mathcal{C}(t)=c(t,\textbf{x})$. The operator $\mathcal{A}(t)$ has a real isolated principal eigenvalue $\mu(\lambda)$ \cite{hess1991periodic}. We aim to obtain error estimates of our Lagrangian method in approximating the principal eigenvalue $\mu(\lambda)$. To this end, we study the approximation of the solution operator for the parabolic equation \eqref{non-autonomousParabolicEQ} by using an operator splitting method.

We define the solution operator $\mathcal{U}(t,s)$ corresponding to the parabolic equation \eqref{non-autonomousParabolicEQ}, which satisfies the following properties:
\begin{enumerate}
	\item $\mathcal{U}(s,s) = Id$, for any $s \ge 0$;
	\item $\mathcal{U}(t,r)\circ \mathcal{U}(r,s) = \mathcal{U}(t,s)$, for any $t \ge r \ge s \ge 0$;
	\item $\frac{d}{dt} \mathcal{U}(t,s) w_0 = \mathcal{A}(t)\mathcal{U}(t,s) w_0$, for any $t\ge s\ge 0, w_0 \in L^2([0,1]^d)$.
\end{enumerate} 
The solution operator $\mathcal{U}(t,s)$ enables us to study the evolution of parabolic operator in  \eqref{non-autonomousParabolicEQ}, e.g., the principal eigenvalue of $\mathcal{U}(T,0)$ gives the principal eigenvalue of the parabolic operator $\mathcal{A}(t)$. It has been proven that the principal eigenvalue of $\mathcal{U}(T,0)$ exists and is real \cite{hess1991periodic}. It is difficult to obtain a closed-form for the solution operator $\mathcal{U}(T,0)$. Therefore, we approximate the solution operator $\mathcal{U}(T,0)$ by using an operator splitting method. 

We set $t_i = i\Delta t$ with $\Delta t = \frac{T}{M}$ and consider the following parabolic equation with freezing time coefficients
\begin{equation}
w_t = \kappa\Delta_{\textbf{x}}  w + \textbf{b}(t_i,\textbf{x})\cdot\nabla_{\textbf{x}} w + c(t_{M-1-i},\textbf{x})w, \quad t_i< t \le t_{i+1}, \quad i\geq 0.\label{freetime1}
\end{equation}
The corresponding solution operator can be formally represented as
\begin{equation}
w(t) = e^{(t-t_i)(\mathcal{L}+ \mathcal{C})(t_i)}\prod_{k = 0}^{i-1}e^{\Delta t(\mathcal{L}+ \mathcal{C})(t_k)}w_0,\quad t_i \le t < t_{i+1}. \label{SoluOperatorfreetime1}
\end{equation}
Furthermore, we can apply the first-order Lie-Trotter operator splitting method to 
approximate the solution operator defined in \eqref{SoluOperatorfreetime1} and obtain 
\begin{equation}
w(t) = e^{(t-t_i)\mathcal{L}(t_i)}e^{(t-t_i) \mathcal{C}(t_i)}\prod_{k = 0}^{i-1}e^{\Delta t\mathcal{L}(t_j)}e^{\Delta t \mathcal{C}(t_j)}w_0, \quad t_i \le t < t_{i+1}. 
\label{LieTrotterAPPSoluOperator}
\end{equation} 

We will prove the solution operator $\prod_{j = 0}^{M-1}e^{\Delta t\mathcal{L}(t_j)}e^{\Delta t \mathcal{C}(t_j)}$ obtained by the Lie-Trotter operator splitting method converges to the solution operator $\mathcal{U}(T,0)$ in certain operator norm as $\Delta t$ approaches zero.  As a  consequence of this convergence result, we can further prove the convergence of the principal eigenvalue associated with these two solution operators.  



To make our paper self-contained, we collect several fundamental results for abstract linear evolution equations by semigroup theory in \ref{Sec:SemigroupTheory}.  We begin with the following lemma, which is as a special case of Theorem 1 in \cite{stewart1974generation}.
\begin{lem}
For any fixed $t$, if $\textbf{b}(t,{\bf x})$ and $c(t,{\bf x})$ are smooth and bounded, then the operator $\mathcal{A}(t)$ defined in \eqref{DefinitionOperatorA} is a strongly elliptic operator on $\mathbb{T}^d$. Moreover, 
$\mathcal{A}(t)$ generates an analytic semigroup $e^{\cdot \mathcal{A}(t)}$ in $L^p(\mathbb{T}^d)$, for all $1 \le p \le \infty $. 
\end{lem}

We will prove that, in our non-autonomous parabolic equation setting, the assumptions made in \ref{Sec:SemigroupTheory} are all satisfied, so we can obtain the error of the operator splitting 
method in approximation the non-autonomous parabolic operator. 

We first prove that the operator $\mathcal{A}$ defined in \eqref{DefinitionOperatorA} satisfies 
a H\"{o}lder continuous condition. 

\begin{lem}\label{Ellipconti}
	Suppose $\textbf{b}(t,\textbf{x})$ and $c(t,\textbf{x})$ in the operator $\mathcal{A}(t)$ are bounded,
	smooth and periodic in each component of $\textbf{x}$, and uniformly H\"{o}lder continuous in $t$, i.e.,
	for any $t,s\in \mathbb{R}^+$, 
	\begin{equation}
	\big|\big|\textbf{b}(t,\textbf{x}) - \textbf{b}(s,\textbf{x})\big|\big| \le C_1 |t-s|^\beta, \quad  \big|c(t,\textbf{x}) - c(s,\textbf{x})\big| \le C_1 |t-s|^\beta,
	\label{bcHolderContinuous}
	\end{equation}
	for some positive $C_1$ and $\beta$.  Let $v \in \mathcal{D}(\mathcal{A}(\cdot))= H^2(\mathbb{T}^d)$ be periodic. Then, for any $0 < s\le \tau$, there exists ${\gamma_1} > 0$, such that  
	\begin{equation}
	\big|\big|\mathcal{A}(\tau)v - \mathcal{A}(s)v\big|\big|_{L^2} \le C_2(\tau-s)^{\beta}\big|\big|(\mathcal{A}(t)-{\gamma_1})v\big|\big|_{L^2}^{1/2} \big|\big|v\big|\big|_{L^2}^{1/2},
	\label{AHolderContinuous1}
	\end{equation} 
	for any $t \in \mathbb{R}^+$. Specifically, if $\textbf{b}(t,\textbf{x}) = 0$, then 
	\begin{equation}
	\big|\big|\mathcal{A}(\tau) v - \mathcal{A}(s) v\big|\big|_{L^2} \le C_3(\tau - s)^\beta \big|\big|v\big|\big|_{L^2}.
	\label{AHolderContinuous2}
	\end{equation} 
\end{lem}
\begin{proof}
	We first consider the case when  $\textbf{b}(t,\textbf{x})\ne 0$. By using the uniformly H\"{o}lder continuous conditions for $\textbf{b}(t,\textbf{x})$ and $c(t,\textbf{x})$, we have 
	\begin{align}
	\big|\big|\mathcal{A}(\tau)v - \mathcal{A}(s)v\big|\big|_{L^2}  
	=& \big|\big|(\textbf{b}(\tau,\textbf{x})-\textbf{b}(s,\textbf{x}))\cdot\nabla_{\textbf{x}} v + (c(\tau,\textbf{x}) - c(s,\textbf{x}))v\big|\big|_{L^2} \nonumber\\
	\le &C_1(t-s)^\beta ( ||\nabla_{\textbf{x}} v||_{L^2} + ||v||_{L^2}).  
	\label{AvAvUpperBound} 
	\end{align}
	For the operator $\mathcal{A}(t)$, we claim that there exists ${\gamma_1} > 0$ such that,
	\begin{equation}
	\big|\big|(\mathcal{A}(t) - {\gamma_1})v\big|\big|_{L^2} > C(\kappa,\textbf{b},c) \big(||\Delta_{\textbf{x}} v||_{L^2} + ||v||_{L^2}\big), \quad \forall v \in \mathcal{D}(\mathcal{A}(\cdot)),
	\label{AvlowerBound}
	\end{equation} 
	where the constant $C(\kappa,\textbf{b},c)$ depends on $\kappa$, $\textbf{b}(t,\textbf{x})$ and $c(t,\textbf{x})$.
	
	We prove the statement in \eqref{AvlowerBound} before move to the main results. Let $c_{\gamma_1} = c-{\gamma_1}$ and assume $\big|\big|\textbf{b}(t,\textbf{x})\big|\big| \le M_1$, $|c(t,\textbf{x})| \le M_2$, and $\big|\big|\nabla_{\textbf{x}} c(t,\textbf{x})\big|\big| \le M_3$. We know that 
	\begin{align}
	\big|\big|(\mathcal{A}(t)-{\gamma_1})v\big|\big|_{L^2} &= \big|\big|(\kappa \Delta_{\textbf{x}} + \textbf{b}(t,\textbf{x})\cdot \nabla_{\textbf{x}} + c_{\gamma_1}(t,\textbf{x}))v\big|\big|_{L^2}\nonumber\\
	&\ge  \big|\big|(\kappa \Delta_{\textbf{x}}  + c_{\gamma_1}(t,\textbf{x}))v\big|\big|_{L^2} - \big|\big| \textbf{b}(t,\textbf{x})\cdot \nabla_{\textbf{x}} v\big|\big|_{L^2}.
	\label{FirstTermEstimate0}
	\end{align}	
	For the term $\big|\big|(\kappa \Delta_{\textbf{x}}  + c_{\gamma_1}(t,\textbf{x}))v\big|\big|_{L^2}$, the periodic condition of $v$ implies that 
	\begin{align}
	&\big|\big|(\kappa \Delta_{\textbf{x}}  + c_{\gamma_1}(t,\textbf{x}))v\big|\big|_{L^2}^2 \nonumber\\
	=& ||\kappa\Delta_{\textbf{x}} v||_{L^2}^2 + \big|\big|c_{\gamma_1}(t,\textbf{x}) v\big|\big|_{L^2}^2 - 2 \langle \kappa \nabla_{\textbf{x}} v, c_{\gamma_1}(t,\textbf{x}) \nabla_{\textbf{x}} v \rangle_{L^2} - 2 \langle \kappa\nabla_{\textbf{x}} v, v\nabla_{\textbf{x}} c(t,\textbf{x}) \rangle_{L^2}.
	\label{FirstTermEstimate1}
	\end{align} 
	Notice that if we choose ${\gamma_1} = \frac{2M_1^2}{\kappa}+M_2$, then we obtain 
	\begin{align}
	-2 \langle \kappa \nabla_{\textbf{x}} v, c_{\gamma_1}(t,\textbf{x}) \nabla_{\textbf{x}} v \rangle_{L^2} \ge 4\kappa({\gamma_1} - M_2) ||\nabla_{\textbf{x}} v||_{L^2}\ge 4\big|\big|\textbf{b}(t,\textbf{x})\cdot \nabla_{\textbf{x}} v\big|\big|_{L^2}.
	\label{FirstTermEstimate2}
	\end{align} 
	In addition, we have 
	\begin{align}
	2 \langle \kappa\nabla_{\textbf{x}} v, v\nabla_{\textbf{x}} c(t,\textbf{x}) \rangle_{L^2} \le 2\kappa M_3  ||\nabla_{\textbf{x}} v||_{L^2} ||v||_{L^2}\le 2 \kappa  M_3 C||\Delta_{\textbf{x}} v||_{L^2}^{\frac{1}{2}} ||v||_{L^2}^{\frac{3}{2}}.
	\label{FirstTermEstimate3}
	\end{align} 
	Here, we use the fact that $||\nabla_{\textbf{x}} v||_{L^2} \le C||\Delta_{\textbf{x}} v||_{L^2}^{\frac{1}{2}} ||v||_{L^2}^{\frac{1}{2}}$, which is the moment inequality in interpolation theory; see Theorem 5.34 of \cite{engel1999one}.
	If we take ${\gamma_1}$ large enough such that $4(\frac{{\gamma_1}-M_2}{3})^\frac{3}{4} \kappa^\frac{1}{4} \ge 2\kappa M_3 C$, we get that
	\begin{equation}
	||\kappa\Delta_{\textbf{x}} v||_{L^2}^2 + \big|\big|c_{\gamma_1}(t,\textbf{x}) v\big|\big|_{L^2}^2 
	\ge 2 \kappa  M_3 C||\Delta_{\textbf{x}} v||_{L^2}^{\frac{1}{2}} ||v||_{L^2}^{\frac{3}{2}}
	\ge  2 \langle \kappa\nabla_{\textbf{x}} v, v\nabla_{\textbf{x}} c(t,\textbf{x}) \rangle_{L^2}.
	\label{FirstTermEstimate4}
	\end{equation}
	Substititing the estiamtes \eqref{FirstTermEstimate2}-\eqref{FirstTermEstimate4} into \eqref{FirstTermEstimate1}, we obtain   
	\begin{equation}
	\big|\big|(\kappa \Delta_{\textbf{x}}  + c_{\gamma_1}(t,\textbf{x}))v\big|\big|_{L^2} \ge 2\big|\big|\textbf{b}(t,\textbf{x})\cdot\nabla_{\textbf{x}} v\big|\big|_{L^2}.
	\label{FirstTermEstimate5}	
	\end{equation} 
	Thus, from \eqref{FirstTermEstimate0} we get that 
	\begin{equation}
	\big|\big|(\mathcal{A}(t)-{\gamma_1})v\big|\big|_{L^2} \ge \frac{1}{2}\big|\big|(\kappa \Delta_{\textbf{x}}  + c_{\gamma_1}(t,\textbf{x}))v\big|\big|_{L^2}.
	\label{FirstTermEstimate6}
	\end{equation}
	Using the same argument, we can prove that for ${\gamma_1}$ large enough, 
	\begin{equation}
	\big|\big|(\kappa \Delta_{\textbf{x}}  + c_{\gamma_1}(t,\textbf{x}))v\big|\big|_{L^2} \ge \hat C (||\Delta_{\textbf{x}}  v||_{L^2} + ||v||_{L^2}).
	\end{equation} 
	Finally, using the moment inequality we prove the statement in \eqref{AHolderContinuous1}.
	
	The case when  $\textbf{b}(t,\textbf{x})= 0$ is simple since we have 
	\begin{equation}
	\big|\big|\mathcal{A}(\tau)v - \mathcal{A}(s)v\big|\big|_{L^2}  = \big|\big| (c(\tau,\textbf{x}) - c(s,\textbf{x}))v\big|\big|_{L^2} \le C_3(t-s)^\beta||v||_{L^2}. 
	\label{AvAvUpperBound2} 	
	\end{equation} 
\end{proof}

We then verify the operators $\mathcal{L}(t)$ and $\mathcal{C}(t)$ defined in \eqref{DefinitionOperatorA} 
satisfy the assumption \ref{commutator}. Given $\tau\ge0$, we assume the bounded conditions as follows
\begin{equation}
||e^{\tau \mathcal{L}(t)}||_{L^2} \le 1,\quad ||e^{\tau \mathcal{C}(t)}||_{L^2} \le 1,\quad ||e^{\tau (\mathcal{L}(t)+\mathcal{C}(t))}||_{L^2} \le 1. \label{LCBounded} 	
\end{equation}
 
\begin{lem}\label{Commu}
	Suppose $\textbf{b}(t,\textbf{x})$ and $c(t,\textbf{x})$ in the operator $\mathcal{A}$ satisfy the same assumption 	as that in Lemma \ref{Ellipconti}. Then, there exists ${\gamma_2}>0$ such that, for any  periodic $v  \in L^2(\mathbb{T}^d)$, commutator of $\mathcal{L}$ and $\mathcal{C}$ acting on $v$ follows,
	\begin{equation}
	\big|\big|[\mathcal{L}(t),\mathcal{C}(t)] v\big|\big|_{L^2} \le C_1 \big|\big| (\mathcal{L}(t)-{\gamma_2}) v\big|\big|^{\frac{1}{2}}_{L^2} ||v||_{L^2}^{\frac{1}{2}}, \quad \forall t\ge 0.
	\label{LCCL} 
	\end{equation}	
\end{lem}
\begin{proof}
	We first observe that, for any $v$ periodic in $L^2(\mathbb{T}^d)$,
	\begin{align}
	&\big|\big|[\mathcal{L}(t),\mathcal{C}(t)] v\big|\big|_{L^2} = \big|\big|\mathcal{L}(t)(\mathcal{C}(t)v)-\mathcal{C}(t)(\mathcal{L}(t)v)\big|\big|_{L^2} \nonumber\\
	= &\big|\big|\big(\kappa \Delta_{\textbf{x}} c(t,\textbf{x}) + \textbf{b}(t,\textbf{x})\cdot \nabla_{\textbf{x}} c(t,\textbf{x})\big) v + 2 \kappa \nabla_{\textbf{x}} c(t,\textbf{x})\cdot\nabla v\big|\big|_{L^2}\nonumber\\
	\le & (\kappa M_4 + M_1M_3)||v||_{L^2} + 2\kappa M_3 ||\nabla_{\textbf{x}} v||_{L^2},
	\end{align}
	where $||b(t,\textbf{x})||\le M_1$, $||\nabla_{\textbf{x}} c(t,\textbf{x})|| \le M_3$, and $|\Delta_{\textbf{x}} c(t,\textbf{x})| \le M_4$.  Following the same procedure as in the proof of Lemma \ref{Ellipconti}, we have
	\begin{align}
	\big|\big|(\mathcal{L}(t) - {\gamma_1})v\big|\big|_{L^2} \ge C(\kappa,\textbf{b}) (||\Delta_{\textbf{x}} v||_{L^2} + ||v||_{L^2}). 
	\label{Lvlowerbound}
	\end{align}
	Using the fact that $||\nabla_{\textbf{x}} v||_{L^2} \le C||\Delta_{\textbf{x}} v||_{L^2}^{\frac{1}{2}} ||v||_{L^2}^{\frac{1}{2}}$, we finally prove the assertion in \eqref{LCCL}.
\end{proof}
\begin{remark}
	If the bounded conditions \eqref{LCBounded} for $\mathcal{L}(t)$ and $\mathcal{C}(t)$ 
	do not hold, we can shift the operators by a constant so that the shifted operators satisfy the bounded condition. Shift the operator by a constant will not affect the commutator in \eqref{LCCL}.
\end{remark}
Now we are in the position to present the main result in approximating the solution operator  $\mathcal{U}(t,s)$ for the parabolic equation \eqref{non-autonomousParabolicEQ}.

\begin{thm} \label{Errorestimate1}
The solution operator \eqref{SoluOperatorfreetime1} has the following error in approximating the solution operator $\mathcal{U}(T,0)$ in $L^2$ operator norm
\begin{align}
\big|\big|\mathcal{U}(T,0) - \prod_{k=0}^{M-1} e^{\Delta t \mathcal{A}(k\Delta t)}\big|\big|_{L^2(\mathbb{T}^d)} \le C_1(T)(\Delta t)^{\beta-\frac{1}{2}},
\label{OperatorApproximation1}
\end{align} 
where $T > 0$, $M$ is an integer, and $\Delta t = \frac{T}{M}$. In addition, the Lie-Trotter operator splitting method has the following error in approximaing the solution operator \eqref{SoluOperatorfreetime1}
\begin{align}
\Big|\Big|\prod_{k=0}^{M-1} e^{\Delta t \mathcal{A}(k\Delta t)} - \prod_{k=0}^{M-1} e^{\Delta t\mathcal{L}(k\Delta t)}e^{\Delta t\mathcal{C}(k\Delta t)}\Big|\Big|_{L^2(\mathbb{T}^d)} \le C_2(T) (\Delta t)^{\frac{1}{2}}.
\label{OperatorApproximation2}
\end{align} 
\end{thm}
\begin{proof}
	We take $\gamma = \max(\gamma_1,\gamma_2)$, where $\gamma_1$ and $\gamma_2$ are defined in Lemma \ref{Ellipconti} and Lemma \ref{Commu} respectively. Let $\mathcal{U}_{\gamma} (t,s) = e^{-{\gamma} (t-s)} \mathcal{U}(t,s)$ be the solution operator that corresponds to the parabolic equation \eqref{non-autonomousParabolicEQ} with  $\mathcal{A}_{\gamma}(t) = \mathcal{A}(t) - {\gamma}$ and $\mathcal{L}_{\gamma}(t) = \mathcal{L}(t) - {\gamma}$. Then, we have 
	\begin{align}
	\mathcal{U}(T,0) - \prod_{k=0}^{M-1} e^{\Delta t \mathcal{A}(k\Delta t)} = e^{{\gamma} T} (U_{\gamma}(T,0) - \prod_{k=0}^{M-1} e^{\Delta t A_{\gamma}(k\Delta t)}).
	\end{align} 
	The statement in \eqref{OperatorApproximation1} is proved according to Theorem \ref{Euler-appendix}.  
	
    For the Lie-Trotter operator splitting method, we know that 
    \begin{align}
    \prod_{k=0}^{M-1} e^{\Delta t \mathcal{A}(k\Delta t)} - \prod_{k=0}^{M-1} e^{\Delta t\mathcal{L}(k\Delta t)}e^{h\mathcal{C}(k\Delta t)} = e^{{\gamma} T}(\prod_{k=0}^{M-1} e^{\Delta t \mathcal{A}_\gamma(k\Delta t)} - \prod_{k=0}^{M-1} e^{\Delta t\mathcal{L}_\gamma(k\Delta t)}e^{\Delta t\mathcal{C}(k\Delta t)}).
    \end{align} 
	Now according to Lemma \ref{Ellipconti} and Lemma \ref{Commu}, $\mathcal{A}_\gamma(k\Delta t) = \mathcal{L}_\gamma(k\Delta t)+\mathcal{C}(k\Delta t)$, $\mathcal{L}_\gamma$, and $\mathcal{C}$ satisfy the assumptions \ref{analytical} and \ref{commutator}. Thus, applying Theorem \ref{onestep} and Theorem \ref{splitting}, we can prove the estimate \eqref{OperatorApproximation2}.
\end{proof}
The convergence of $\mathcal{K}^{\Delta t}$ in the operator norm  $\mathcal{L}(L^2,H^1)$ has been proved in \cite{batkai2011operator}. 
In Theorem \ref{Errorestimate1}, we obtain the convergence of $\mathcal{K}^{\Delta t}$
in the operator norm $\mathcal{L}(L^2)$. Finally, we can obtain the error estimate for the principal eigenvalue. 
\begin{thm} \label{Errorestimate2}
	Let $e^{\mu(\lambda)T}$ and $e^{\mu_{\Delta t}(\lambda)T}$ denote the principal eigenvalue of the solution operator $\mathcal{U}(T,0)$ and the approximated solution operator $\prod_{k=0}^{M-1} e^{\Delta t\mathcal{L}(k\Delta t)}e^{\Delta t\mathcal{C}(k\Delta t)}$, respectively. Then, we have the error estimate as follows:
	\begin{align}
	\big|e^{\mu(\lambda)T}- e^{\mu_{\Delta t}(\lambda)T}\big| \le C_1(T)(\Delta t)^{\beta-\frac{1}{2}}+C_2(T) (\Delta t)^{\frac{1}{2}}.
	\label{OperatorApproximation3}
	\end{align}  
	Moreover, we can obtain that $|\mu(\lambda)-\mu_{\Delta t}(\lambda)|={O}\big((\Delta t)^{\min(\beta - \frac{1}{2},\frac{1}{2})}\big)$.
\end{thm}
\begin{proof}
	According to the standard spectral theorem \cite{kato2013perturbation}, the principal eigenvalue $e^{\mu(\lambda)}$ of the solution operator $\mathcal{U}(T,0)$ and the principal eigenvalue $e^{\mu_{\Delta t}(\lambda)}$ of the approximated solution operator $\prod_{k=0}^{M-1} e^{\Delta t\mathcal{L}(k\Delta t)}e^{\Delta t\mathcal{C}(k\Delta t)}$ satisfy
	\begin{align}
	\big|e^{\mu(\lambda)T}- e^{\mu_{\Delta t}(\lambda)T}\big| \le C_3\big|\big|U(T,0) - \prod_{k=0}^{N-1} e^{\Delta t\mathcal{L}(k\Delta t)}e^{\Delta t\mathcal{C}(k\Delta t)}\big|\big|_{L^2(\mathbb{T}^d)}. \label{OperatorApproximation4}
	\end{align} 
	By using the triangle inequality for the right hand side of \eqref{OperatorApproximation4} 
	and the estimated results from Theorem \ref{Errorestimate1}, we can get the error estimate 
	\eqref{OperatorApproximation3}. The error estimate for $|\mu(\lambda)-\mu_{\Delta t}(\lambda)|$ can be obtained accordingly.
\end{proof}	
In this paper, we assume that $\textbf{b}(t,\textbf{x})$ and $c(t,\textbf{x})$ in the operator $\mathcal{A}$ are uniformly Lipschitz. Thus, the error of the principal eigenvalue obtained by the Lie-Trotter operator splitting method
is at least $O((\Delta t)^{\frac{1}{2}})$.

\subsection{Analysis of the Lagrangian particle method}\label{sec:AnalysisLagrangianMethod}
\noindent
We consider the Feynman-Kac semigroup $\Phi^{\mathcal{A}}$ associated with an arbitary operaor $\mathcal{A}$. The action of the Feynman-Kac semigroup $\Phi^{\mathcal{A}}$ on a probability measure
$\nu$ is defined by 
\begin{equation}
\Phi^{\mathcal{A}} (\nu)(\phi) = \frac{ (\nu)(\mathcal{A}\phi)}{(\nu)(\mathcal{A}1)}, \quad
\forall \phi\in L^2(\mathbb{T}^d). \label{DefineFKsemigroup}
\end{equation}
Moreover, we denote $\Phi^{\mathcal{A}}_{n}=(\Phi^{\mathcal{A}})^{n}$. 
The Feynman-Kac semigroup operation satisfies the following property. 
\begin{lem} \label{multible}
	For any operaors $\mathcal{A}$, $\mathcal{B}$ in $\mathcal{L}(L^2(\mathbb{T}^d))$,
	$\Phi^{\mathcal{A}\mathcal{B}} = \Phi^{\mathcal{B}} \Phi^{\mathcal{A}}.$
\end{lem}
\begin{proof}
	Let $\nu$ be a probability measure and $\phi$ be a function in $L^2(\mathbb{T}^d)$. Then, we 
	can easily verify that  
	\begin{align}
	\Phi^{\mathcal{A}\mathcal{B}}(\nu)(\phi) &= \frac{ (\nu)(\mathcal{A}\mathcal{B}\phi)}{(\nu)(\mathcal{A}\mathcal{B}1)} =  \frac{ (\nu)(\mathcal{A}\mathcal{B}\phi)}{(\nu)(\mathcal{A} 1)} \frac{(\nu)(\mathcal{A} 1)}{(\nu)(\mathcal{A}\mathcal{B} 1)} \nonumber \\
	&=\frac{\Phi^{\mathcal{A}}(\nu)(\mathcal{B} \phi)}{\Phi^{\mathcal{A}}(\nu)(\mathcal{B} 1)} = \Phi^{\mathcal{B}} \Phi^{\mathcal{A}} (\nu) (\phi).
	\end{align}
\end{proof}
Recall that the operator $\Phi_n^{\mathcal{K}^{\Delta t}}$ defined in \eqref{operatorPhiKDeltat} is a compostion of the Feynman-Kac semigroup $\Phi^{\mathcal{K}^{\Delta t,i}}$ associated with the operator $\mathcal{K}^{\Delta t,i}$; see \eqref{operatorPhiKiDeltat}. In the sequel, we will prove the operator $\Phi_n^{\mathcal{K}^{\Delta t}}$ satisfies the uniform minorization and boundedness condition, which guarantees the existence of an invariant measure.  
\begin{thm}\label{bam}
There exists a probability measure $\eta$ so that the operator $\mathcal{K}^{\Delta t}$ satisfies a uniform minorization and boundedness condition as follows
\begin{equation} 
\epsilon \eta(\phi) \le \mathcal{K}^{\Delta t} (\phi) (\textbf{x}) \le \gamma \eta(\phi), 
\quad \forall \textbf{x} \in \mathbb{T}^d, \quad \forall \phi\in L^2(\mathbb{T}^d),
\label{MinorizationBoundedCondition}
\end{equation}
where $0<\epsilon<\gamma $ are independent of $\Delta t$.
Moreover, when $\Delta t \to 0$ the limit operator is the exact solution operator $\mathcal{U}(T,0)$, which also satisfies this uniform minorization and boundedness condition.
\end{thm}
\begin{proof}
We first define an operator $P^{\Delta t} = \prod_{i=0}^{M-1} P_{t_i}^{\Delta t}$, which corresponds to the case when $c(t,\textbf{x}) = 0$ in Eq.\eqref{non-autonomousParabolicEQ}. Since $c(t,\textbf{x})$ is bounded (i.e. $c_1 \le c(t,\textbf{x}) \le c_2$), one can easily obtain the following estimate based on the Feynman-Kac formula   
\begin{equation}
P^{\Delta t}(\phi)e^{c_1T} \le \mathcal{K}^{\Delta t}(\phi) \le P^{\Delta t} (\phi)e^{c_2T}. 
\end{equation}
Thus, to estimate the bounds for $\mathcal{K}^{\Delta t}$, we only need to study the operator  $P^{\Delta t}$. Moreover, it is sufficient to prove that there exist a probability measure $\eta$ and a constant $\epsilon > 0$
so that for any indicator function of a Borel set $S \subset \mathbb{T}^d$ the following result holds 
\begin{equation}
\mathbb{P}({\bf X}_M \in S| {\bf X}_0 = \textbf{x}) \ge \epsilon \eta(S),
\end{equation} 
where ${\bf X}_i$ are defined in the  scheme \eqref{NumericalMethod4SDE} as the numerical solution to the SDE \eqref{SDE-system}. The idea of the proof is to explicitly rewrite ${\bf X}_M$ as a perturbation of the reference evolution corresponding to $\textbf{b}  = 0$. According to the numerical scheme \eqref{NumericalMethod4SDE}, we have 
\begin{equation}\label{split}
{\bf X}_M = {\bf X}_0 + \bm G_M + \bm F_M, 
\end{equation} 
where 
\begin{equation}
\bm G_M = \sqrt{2\kappa\Delta t}\sum_{i=0}^{M-1}\bm \omega_i, \quad \text{and}, \quad 
\bm F_M =\Delta t\sum_{i=0}^{M-1} \textbf{b}(i\Delta t, {\bf X}_i).
\end{equation} 
We know that $|\bm F_M| \le T||\textbf{b}||_{L^\infty}$ and $\bm G_M$ is a Gaussian random variable with covariance matrix $2\kappa T \text{I}_d$, where $\text{I}_d$ is the $d$-dimensional identity matrix. Therefore  
\begin{align}
\mathbb{P}({\bf X}_M \in S| {\bf X}_0 = \bm x) &\ge \mathbb{P}(\bm G_M \in S- \bm x-\bm F_M) \nonumber \\
		&=(\frac{1}{2\pi \kappa T})^{d/2} \int_{S-\bm x-\bm F_m} \exp(-\frac{|\bm y|^2}{2\kappa T}) d\bm y.
\end{align}
Since the state space $\mathbb{T}^d$ is compact, we can find $R >0$ such that $|\bm x + F_M| \le R$ for all $\bm x \in \mathbb{T}^d$. Thus, we define the probability measure $\eta$ as 
\begin{align}
\eta(S) = Z_R^{-1} \inf_{|Q| \le R} \int_{S+Q} \exp (-\frac{|\bm y|^2}{2\kappa T}) d\bm y,\quad 
\forall S \subset \mathbb{T}^d,
\end{align}
where $Z_R$ is the normalization constant. Setting $\epsilon =  Z_R(4\pi\kappa T)^{-d/2}$, 
we can easily verify that  $\eta(S)\ge Z_R^{-1} \exp(-\frac{|R+1|^2}{2\kappa T})|S|$, which satisfies
a uniform minorization condition. 

The uniform boundedness condition is  automatically satisfied since  $\eta$ has a positive density with respect to Lebesgue measure. 

The situation when the exact solution operator is considered can be proved by changing Eq.\eqref{split} into an Ito integration form
\begin{align}
\textbf{X}^{0,\textbf{x}}_t = \textbf{X}^{0,\textbf{x}}_0+ \int_{0}^{t}\textbf{b}(s,\textbf{X}^{0,\textbf{x}}_s)ds 
+ \int_{0}^{t} \sqrt{2\kappa}d\textbf{w}(s) 
\end{align}
and then go through the same procedure.
\end{proof}
 
We now represent an important result that ensures the existence of the limiting measure for the 
discretized Feymann-Kac dynamics. The detailed proof of Theorem \ref{thm:existenceinvariantmeasure} can be found in \cite{Meyn:92} or Corollary 2.5 in \cite{moral2001on}. 
\begin{thm}\label{thm:existenceinvariantmeasure}
Suppose the minorization and boundedness conditions \eqref{MinorizationBoundedCondition} hold true. Then,  $\Phi_n^{\mathcal{K}^{\Delta t}}$ admits an invariant measure $\nu_{\Delta t}$, whose density function is the eigenfunction of the operator $(\mathcal{K}^{\Delta t})^\star$, the adjoint operator of the solution operator $\mathcal{K}^{\Delta t}$. Moreover, for any initial distribution  $\nu_0\in \mathcal{P}(\mathbb{T}^d)$, we have 
\begin{equation}
\big|\big|\Phi_n^{\mathcal{K}^{\Delta t}} (\nu_0) - \nu_{\Delta t}\big|\big|_{TV} \le 2(1-\frac{\epsilon}{\gamma} )^n,
\label{CovInvariantMeasure}
\end{equation}
where $||\cdot||_{TV}$ is the total variation norm and $0<\epsilon<\gamma $ are the parameters defined in the  minorization and boundedness conditions in \eqref{MinorizationBoundedCondition}. The estimate \eqref{CovInvariantMeasure} is also true when changing $\mathcal{K}^{\Delta t}$ to the exact solution operator $\mathcal{U}(T,0)$.
\end{thm}

\begin{cor} The principal eigenvalue $\mu_{\Delta t}$ of $\mathcal{K}^{\Delta t}$ satisfies the following relation 
\begin{equation}
e^{\mu_{\Delta t}(\lambda)T} = \nu_{\Delta t} \mathcal{K}^{\Delta t} 1 = \Phi_n^{\mathcal{K}^{\Delta t}} (\nu_0) \mathcal{K}^{\Delta t} 1 + \rho_n,
\end{equation}
where $\nu_0$ is any bounded non-negative initial probability measure, $T$ is the period of the time parameter, and $\rho_n=O(1-\frac{\epsilon}{\gamma})^n$.
\end{cor}
\begin{proof}
Theorem \ref{thm:existenceinvariantmeasure} implies that for any bounded non-negative measure $\nu_0$, the measure $\Phi_n^{\mathcal{K}^{\Delta t}}(\nu_0)$ converges to an invariant measure $\nu_{\Delta t}$ in the weak sense, that is
\begin{equation}
\nu_{\Delta t} \phi := \int_{\mathbb{T}^d} \phi d\nu_{\Delta t} = \Phi_n^{\mathcal{K}^{\Delta t}} (\nu_0)(\phi) + O(1-\frac{\epsilon}{\gamma})^n,
\end{equation}
for any bounded non-negative measurable function $\phi$.  Then, we take $\phi = \mathcal{K}^{\Delta t} 1$. From the fact that the density function of  $\nu_{\Delta t}$ is the eigenfunction of the operator $(\mathcal{K}^{\Delta t})^\star$, we get that 
\begin{equation} 
\nu_{\Delta t} (\mathcal{K}^{\Delta t} 1) =  ((\mathcal{K}^{\Delta t})^\star \nu_{\Delta t}) 1 = e^{\mu_{\Delta t}(\lambda)T} (\nu_{\Delta t} 1) = e^{\mu_{\Delta t}(\lambda)T}.
\end{equation}
Thus, we finish the proof.
\end{proof}

Now we compute the principal eigenvalue $\mu_{\Delta t}(\lambda)$.
\begin{lem}\label{eig-estimate}
Denote $\nu_{\Delta t}^k = \prod_{i=0}^{k-1} \Phi^{\mathcal{K}^{\Delta t,i}} \nu_{\Delta t}$, $1\leq k \leq M$. Let $e_k = (\nu_{\Delta t}^k) (\mathcal{K}^{\Delta t,k} 1)$ denote the changing of mass. Then, we have 
\begin{equation} 
e^{\mu_{\Delta t}(\lambda)T} = \prod_{k=0}^{M-1} e_k, \quad \text{and} \quad  \mu_{\Delta t}(\lambda) = \frac{1}{M\Delta t} \sum_{k=0}^{M-1} \log(e_k). 
\label{formula4mu}
\end{equation} 
\end{lem}
\begin{proof}
 It is easy to verify that 
\begin{align}
\nu_{\Delta t}^{M} = \nu_{\Delta t}^0 = \nu_{\Delta t},\quad 
(\mathcal{K}^{\Delta t,k})^\star\nu_{\Delta t}^k = e_k\nu_{\Delta t}^{k+1},
\end{align}
for some positive numbers $e_k$'s. These $e_k$'s are referred to as the changing of the mass for each small step $\mathcal{K}^{\Delta t,k}$. Thus, we have $(\mathcal{K}^{\Delta t})^\star\nu_{\Delta t}^0 = (\prod_{k=0}^{M-1} e_k)\nu_{\Delta t}^0$, which means $e^{\mu_{\Delta t}(\lambda)T} = \prod_{k=0}^{M-1} e_k$. By taking the logarithm, we obtain the formula for $\mu_{\Delta t}(\lambda)$ in \eqref{formula4mu}, where $M\Delta t=T$.
\end{proof}
Finally, we show the error estimate of the Lagrangian particle method in computing the principal eigenvalue of parabolic operators as follows. 
\begin{thm}
Suppose $\textbf{b}(t,\textbf{x})$ and $c(t,\textbf{x})$ in $\mathcal{A}(t)$ \eqref{DefinitionOperatorA} are bounded, smooth and periodic in each component of $\textbf{x}$, and uniformly H\"{o}lder continuous in $t$. Let $\mu^n_{\Delta t}(\lambda)=	(M\Delta t)^{-1}\sum_{k=0}^{M-1} \log\big( N^{-1}\sum_{p=1}^{N} 
\exp(c(t_{M-1-k},\widetilde\xi_{k}^{p,n-1})\Delta t)\big)$ denote the approximate principal eigenvalue obtained by the $N$-IPS method, where $\widetilde\xi_{k}^{p,n-1}$, $k = 0,\cdots,M-1$, $p = 1,\cdots,N$, $n$ is the iteration number and $\Delta t$ are defined in the Algrithm \ref{Algorithm-TimeFependent}. Let $\mu(\lambda)$ denote the principal eigenvalue of \eqref{KPPTimePeriodic} defined in Eq.\eqref{Feynman-Kacformula}. Then, we have the following convergence result
\begin{equation}
\lim_{N \to \infty} (M\Delta t)^{-1}\sum_{k=0}^{M-1} \log\Big( N^{-1}\sum_{p=1}^{N} 
\exp(c(t_{M-1-k},\widetilde\xi_{k}^{p,n-1})\Delta t)\Big) = \mu(\lambda) + O\big((1-\frac{\epsilon}{\gamma})^n\big) + O\big((\Delta t)^{\frac{1}{2}}\big),
\label{mainresult-NIPSPincipaleig}
\end{equation}
where $0<\epsilon<\gamma $ are the parameters defined in the  minorization and boundedness conditions in \eqref{MinorizationBoundedCondition}.
\end{thm}

\begin{proof}
By the converence property of the $N$-IPS, we know that the empirical distribution of the particles $\{\widetilde\xi_{k}^{p,n-1} \}_{p = 1,\cdots,N}$ will weakly converge to the distribution $\prod_{i=0}^{k-1} \Phi^{\mathcal{K}^{\Delta t,i}}    \Phi_{n-1}^{\mathcal{K}^{\Delta t}}  \nu_0$, $1\leq k \leq M$,  when $N \to \infty$. Let $e_{k,n}^{N}=N^{-1}\sum_{p=1}^{N} \text{exp}(c(t_{M-1-k},\widetilde\xi_{k}^{p,n-1})\Delta t)$ denote the increasing of the mass for each small step $\mathcal{K}^{\Delta t,k}$. Then, we can get that $\prod_{i=0}^{k-1} \Phi^{\mathcal{K}^{\Delta t,i}}    \Phi_{n-1}^{\mathcal{K}^{\Delta t}}  \nu_0$, $1\leq k \leq M$, satisfy 
\begin{align}
\lim_{N\to \infty} e_{k,n}^{N} = (\prod_{i=0}^{k-1} \Phi^{\mathcal{K}^{\Delta t,i}}    \Phi_{n-1}^{\mathcal{K}^{\Delta t}}  \nu_0) (\mathcal{K}^{\Delta t,k} 1).
\end{align} 
According to Theorem \ref{thm:existenceinvariantmeasure}, we have that $\Phi_{n-1}^{\mathcal{K}^{\Delta t}} \nu_0 =  \nu_{\Delta t} + \delta_n$, where $||\delta_n||_{TV} \le 2(1-\frac{\epsilon}{\gamma} )^n$. 
This implies that 
\begin{align}
\lim_{N\to \infty} e_{k,n}^{N} = (\prod_{i=0}^{k-1} \Phi^{\mathcal{K}^{\Delta t,i}}   \nu_{\Delta t}) (\mathcal{K}^{\Delta t,k} 1) + O\big((1-\frac{\epsilon}{\gamma})^n\big). 
\end{align} 
Combining Lemma \ref{eig-estimate}, we conclude that 
\begin{align}
\lim_{N\to\infty}(M\Delta t)^{-1}\sum_{i=1}^M \log( e_{k,n}^N) &= (M\Delta t)^{-1} \sum_{k=0}^{M-1} \log(e_k) + O\big((1-\frac{\epsilon}{\gamma})^n\big) \nonumber \\
&= \mu_{\Delta t}(\lambda) + O\big((1-\frac{\epsilon}{\gamma})^n\big). 
\end{align} 
From Theorem \ref{Errorestimate2}, we know that $\big|\mu(\lambda)-\mu_{\Delta t}(\lambda)\big|={O}\big((\Delta t)^{\frac{1}{2}}\big)$. Therefore, the estimate in \eqref{mainresult-NIPSPincipaleig} can be obtained by using the triangle inequality.
\end{proof}

\section{Numerical results}\label{sec:NumericalResults}
\noindent
In this section, we first present numerical examples to verify the convergence analysis of the proposed method in computing eigenvalues. Then, we compute the KPP front speeds in 2D and 3D chaotic flows. In addition, we investigate the dependence of the KPP front speed on the magnitude of velocity fields and the evolution of the empirical distribution of the $N$-IPS. To be consistent with the setting of numerical experiments in the literature, e.g., \cite{shen2013finite2d,shen2013finite3d}, we choose the torus space $\mathbb{T}^d=[0,2\pi]^d$, $d=2,3$.  
\subsection{Convergence tests in computing principal eigenvalue}\label{sec:NormConvergeTest}
\noindent 
We first verify the convergence of the operator splitting method in approximating solution operator. Let $\textbf{x}=(x_1,x_2)^{T}$. We consider a two-dimensional non-autonomous equation on $[0,2\pi]^2$ as follows:
\begin{equation}
u_t = \mathcal{L}(t) u + \mathcal{C}(t)u,
\label{NumSec1E1}
\end{equation}
where  $\mathcal{L}(t) = \Delta_{\textbf{x}} + (\sin(x_2)\cos(2\pi t), \sin(x_1)\cos(2\pi t))\cdot \nabla_{\textbf{x}}$ and $\mathcal{C}(t) = \big(\sin(x_1+x_2)+ \cos(x_1+x_2)\big)\sin(2\pi t)$.

We use spectral method to discretize Eq.\eqref{NumSec1E1}, in order to obtain an accurate approximation in the physical space of the solution operator of Eq.\eqref{NumSec1E1}. Speficially, let $V_{H}=\text{span}\{e^{i(k_1x_1 + k_2 x_2)}: -H\leq k_1,k_2\leq H\}$ denote a finite dimensional space spanned by fourier basis functions, where $H$ is a positive integer. First, we compute the approximations of the 
operators $\mathcal{L}(t)$ and $\mathcal{C}(t)$ in the space $V_{H}$. Let matrices $L^H(t)$ and $C^H(t)\in \mathbb{C}^{(2H+1)^2\times(2H+1)^2}$ denote the approximations of $\mathcal{L}(t)$ and $\mathcal{C}(t)$, respectively \cite{shen2011spectral}.
Then, we use the matrix exponential functions $e^{\Delta t L^H(t)}$ and $e^{\Delta t C^H(t)}$ to approximate $e^{\Delta t \mathcal{L}(t)}$ and $e^{\Delta t \mathcal{C}(t)}$, respectively. Thus, we get an approximation formula for $\mathcal{K}^{\Delta t}$ as
\begin{equation} 
K^{H,\Delta t} = \prod_{j = 0}^{T/\Delta t -1}e^{\Delta t L^H(t_j)}e^{\Delta t M^H(t_j)}.
\label{spectrumapprox}
\end{equation}
For the reference solution, we choose a much finer time step $\Delta t_{ref}$ and compute the approximation formula
\begin{equation}
\widetilde K^{H,\Delta t_{ref}} = \prod_{j=0}^{T/\Delta t_{ref}-1} e^{\Delta t_{ref}(L^H(t_j)+M^H(t_j))}. 
\label{spectrumapproxref}
\end{equation}
 
In this experiment, we choose $H = 24$, $\Delta t = 2^{-1}, 2^{-2},\cdots, 2^{-9}$, and $\Delta t_{ref} = 2^{-12}$. Then, we compute $||K^{H,\Delta t} - \widetilde K^{H,\Delta t_{ref}}||_{L^2}$ to verify our result.
Figure \ref{splittingerror} shows the convergence results for the splitting method. The convergence rate is $(\Delta t)^{1.05}$. This numerical result suggests that the convergence analysis in Theorem \ref{Errorestimate2} is not sharp. More studies on the convergence analysis of our method will be reported in our future work. 
 
\begin{figure}[tbph]
	\centering
	\includegraphics[width=0.5\linewidth]{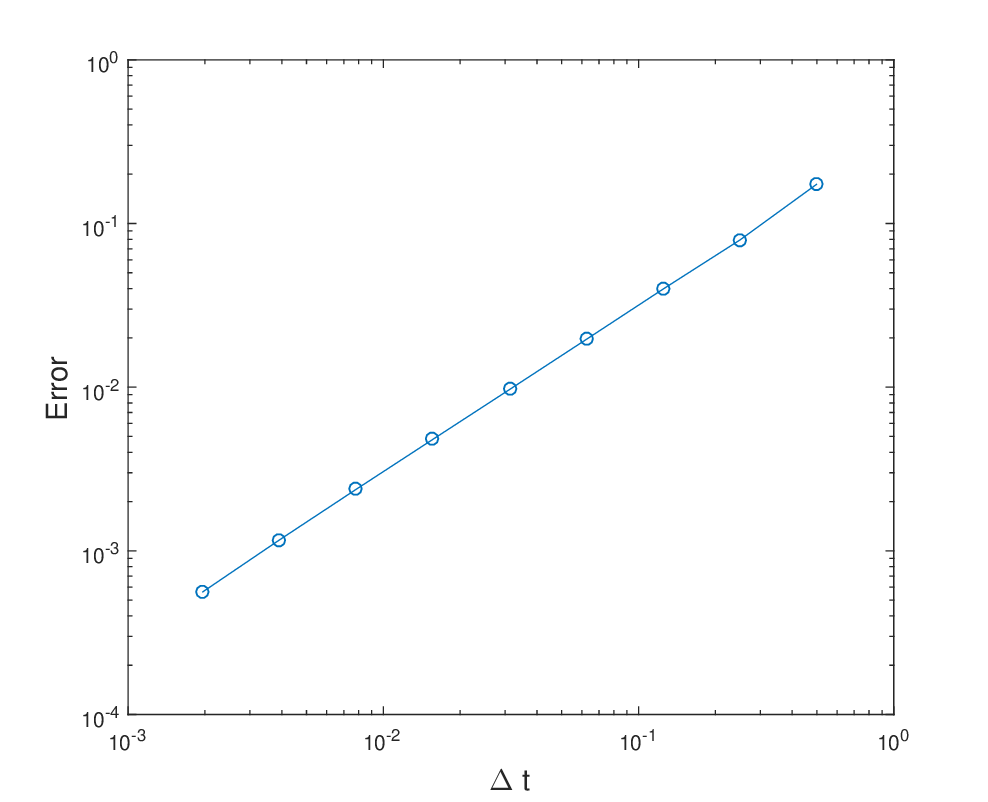}
	\caption{Numerical errors for $||K^{H,\Delta t} - \widetilde K^{H,\Delta t_{ref}}||_{L^2}$.}
	\label{splittingerror}
\end{figure}
Then, we test the convergence of the Lagrangian method, i.e., Algorithm \ref{Algorithm-TimeFependent}, in computing principal eigenvalues of parabolic-type equations. 
We still consider the problem \eqref{NumSec1E1} with the same $\mathcal{L}(t)$ and $\mathcal{C}(t)$.
In this experiment, we choose $\Delta t = 2^{-1},2^{-2},2^{-3},2^{-4},2^{-5}$,  $N = 200,000$ in the $N$-IPS system, and iteration number $n = 200$ and $n=400$ in the Feynman-Kac semigroup iteration method. Figure \ref{fig:proberror} shows the convergence of principal eigenvalues with respect to $\Delta t$ by spectral method and our Lagrangian method, where the reference solution is computed from spectral method with a finer grid $\Delta t_{ref} = 2 ^{-10}$. So given sufficient large $N$ and $n$, the error in calculating principal eigenvalues of linearized KPP operator $\mathcal{A}$ via our proposed Lagrangian approach only comes from the error of operator splitting. Also as the Lagrangian method will eventually converge to some invariant measure approximating the ground truth invariant measure, there is no error accumulation for long-time integration.

\begin{figure}[tbph]
	\centering
	\includegraphics[width=0.5\linewidth]{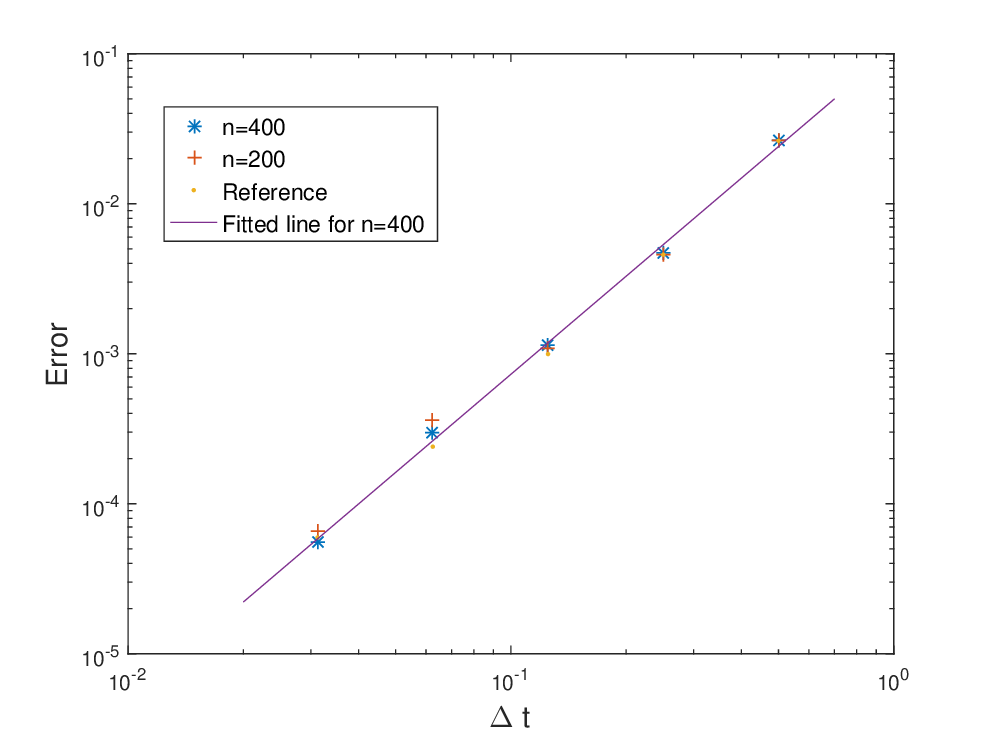}
	\caption{In the Lagrangian method, iteration number $n=200$ and $n=400$. The reference solution is obtained by the spectral method.}
	\label{fig:proberror}
\end{figure}

\subsection{Computing KPP front speeds in different flows}\label{sec:KPPFrontSpeedsFlows}
\noindent
We first compute the KPP front speeds in two different time-independent flows, i.e., a 2D steady cellular flow
and a 3D ABC flow. Let $\textbf{x}=(x_1,...,x_d)^{T}\in [0,2\pi]^d$ with $d=2,3$. We use the Lagrangian method to compute the following principal eigenvalue problem  with periodic boundary condition 
\begin{equation}
\begin{aligned}
&\kappa \Delta_{\textbf{x}} \Phi + (2\kappa {\lambda} {\bf e + v})\cdot \nabla_{\textbf{x}} \Phi + \big(\kappa {\lambda}^2 + {\lambda}{\bf e\cdot v} + \tau^{-1} f'(0) \big)\Phi = \mu(\lambda) \Phi, 
\end{aligned}
\label{KPPspeed2Dsteadyflow}
\end{equation}
 $f(u)=u(1-u)$, and $(\mu(\lambda),\Phi)$ are principal eigenvalue of \eqref{KPPspeed2Dsteadyflow} and its associated eigenfunction, respectively. The velocity field ${\bf v} = (- \sin x_1 \cos x_2, \cos x_1 \sin x_2)$ in the 2D steady cellular flow and ${\bf v} = (\sin x_3 + \cos x_2, \sin x_1 + \cos x_3, \sin x_2 + \cos x_1)$ in the 3D ABC flow, respectively.  

We choose the parameters $\kappa = 1$ and $\tau = 1$ in \eqref{KPPspeed2Dsteadyflow}. We use the spectral method to obtain an accurate reference solution for the principal eigenvalue of \eqref{KPPspeed2Dsteadyflow}.    
Figure \ref{fig:detconvergence} shows the convergence results of the Lagrangian method in computing the principal eigenvalue, where $\lambda = 0.35$ for the 2D cellular flow and $\lambda = 0.55$ for the 3D ABC flow.
We find the convergence rate of the Lagrangian method is $(\Delta t)^{1.51}$ for the 2D steady cellular flow,
and $(\Delta t)^{1.70}$ for the 3D ABC flow. Thus, we can use the Lagrangian method to compute the KPP front speeds in both 2D and 3D flows. 
 
\begin{figure}[tbph]
	\centering
	\begin{subfigure}{0.45\textwidth}
		\includegraphics[width = \linewidth]{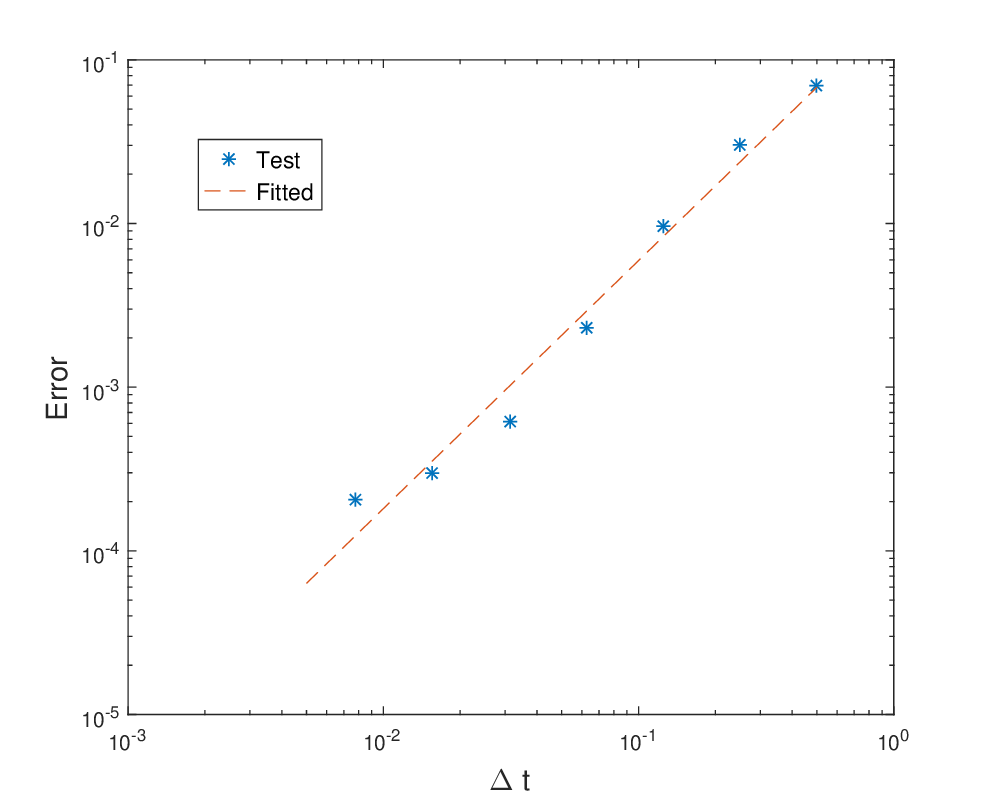}
		\caption{2D convergence test, fitted slope $\approx$ 1.51}
		\label{convergece-dt-2D} 
	\end{subfigure}
	\begin{subfigure}{0.45\textwidth}
		\includegraphics[width = \linewidth]{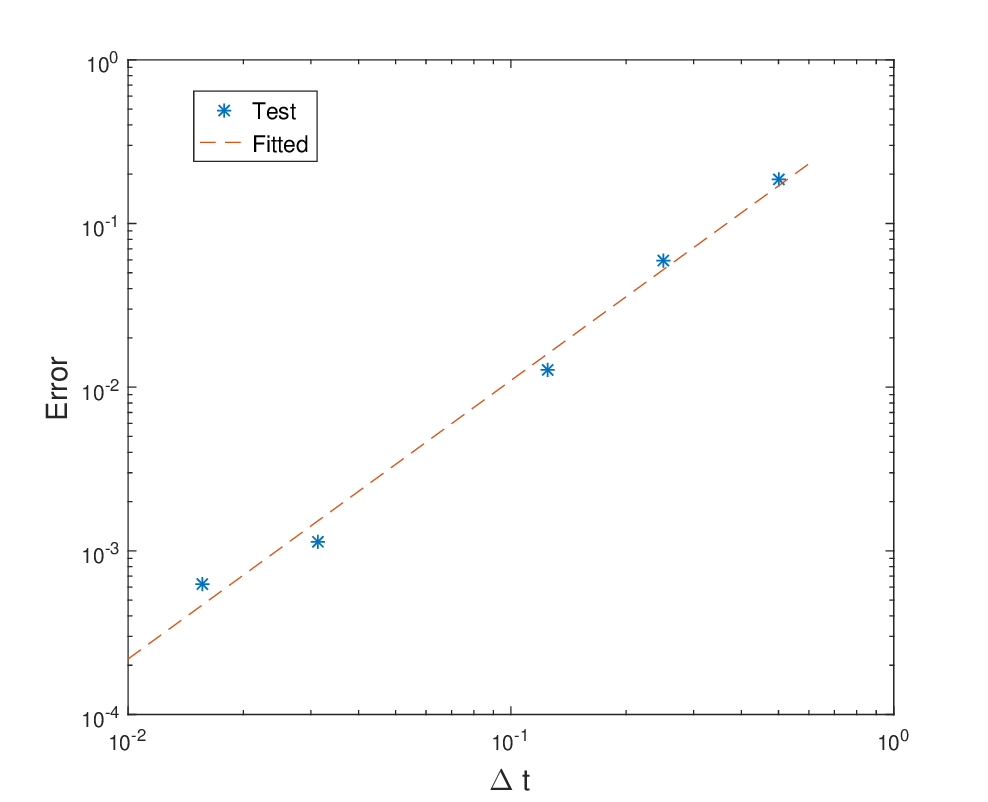}
		\caption{3D convergence test, fitted slope $\approx$ 1.60.}
		\label{convergece-dt-3D}
	\end{subfigure}
	\caption{Errors of the principal eigenvalue computed by using different time steps.}
	\label{fig:detconvergence}
\end{figure}   

After getting the principal eigenvalue, we compute the KPP front speed $c^*$ through the formula $c^* = \inf_{{\lambda} > 0} \frac{\mu(\lambda)}{\lambda}$.
We only show the numerical results for the 3D ABC flow here since the results for the 2D steady cellular flow is quantitatively similar. We choose the velocity field  ${\bf v} = A(\sin x_3 + \cos x_2, \sin x_1 + \cos x_3, \sin x_2 + \cos x_1)$, where $A$ is the strength of the convection.  In Figure \ref{fig:difA}, we show the results of $\frac{\mu(\lambda)}{\lambda}$ for ABC flows with $A = 1$ and $A = 10$. The amplitude of the principal eigenvalue increases fast and the convergence speed becomes slower. Notice that in this case, the flow becomes very unstable since the convection becomes dominant comparing to the diffusion. This issue will be studied in subsection \ref{sec:NumCompKPPspead}.


\begin{figure}[tbph]
	\centering
	\begin{subfigure}{0.45\textwidth}
		\includegraphics[width = \linewidth]{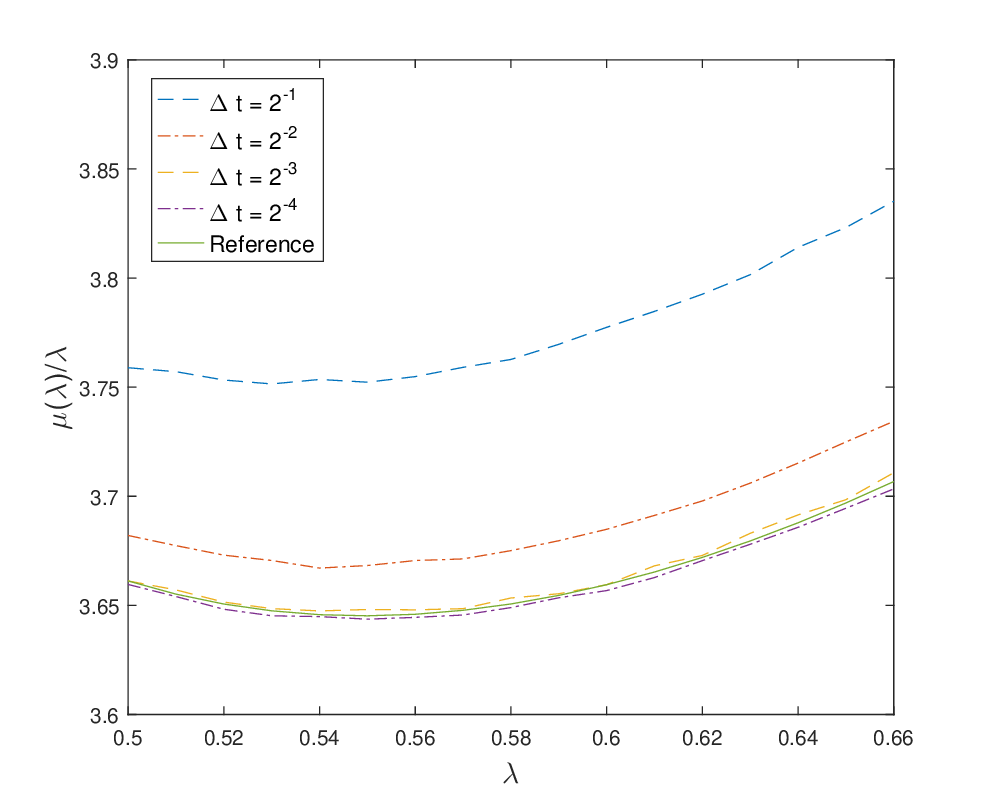}
		\caption{A = 1 }
		\label{convergece-dt-2D} 
	\end{subfigure}
	\begin{subfigure}{0.45\textwidth}
		\includegraphics[width = \linewidth]{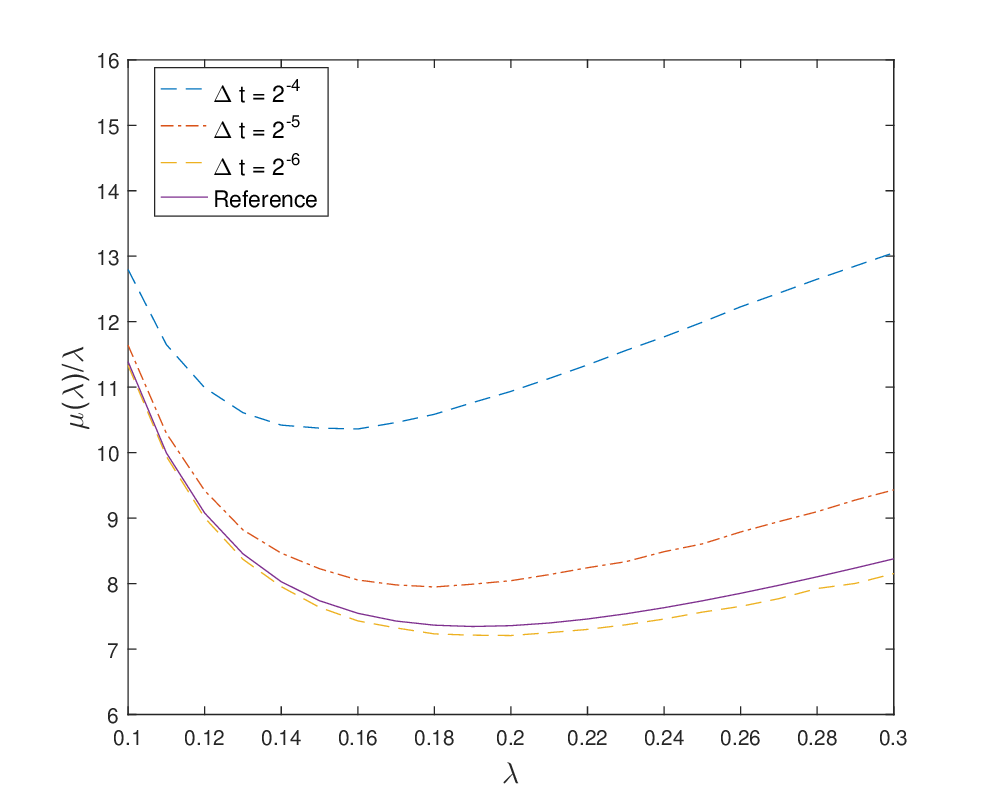}
		\caption{ A = 10 }
		\label{convergece-dt-3D}
	\end{subfigure}
	\caption{Numerical results of $\frac{\mu({\lambda})}{{\lambda}}$ for different $\lambda$'s in the ABC flow.}
	\label{fig:difA}
\end{figure}

Next, we compute the KPP front speed in a 2D unsteady (time-dependent) cellular flow. Let $\textbf{x}=(x_1,x_2)^{T}$. We use the Lagrangian method to compute the following principal eigenvalue problem  with periodic boundary condition 
\begin{equation}
\begin{aligned}
& \kappa \Delta_{\textbf{x}} \Phi + (2\kappa \lambda {\bf e} + {\bf v})\cdot \nabla_{\textbf{x}} \Phi + \big(\kappa {\lambda}^2 + {\lambda}{\bf e\cdot v} + \tau^{-1} f'(0) \big)\Phi-\Phi_{t} = \mu(\lambda) \Phi,  
\end{aligned}
\label{KPPspeed2Dunsteadyflow}
\end{equation}
where $ (t,\textbf{x})\in [0,T]\times [0,2\pi]^2$, $T$ is the period of $\textbf{v}$ in $t$,  $f(u)=u(1-u)$, and $(\mu(\lambda),\Phi)$ are principal eigenvalue of \eqref{KPPspeed2Dunsteadyflow} and its associated eigenfunction, respectively. The velocity field of the 2D unsteady cellular flow is ${\bf v} = \big(-\sin x_1 \cos x_2 (1+\delta \cos 2\pi t),\cos x_1 \sin x_2 (1+\delta \cos 2\pi t)\big)$, where $\delta>0$ is a parameter.  

We choose the parameters $\kappa = 1$ and $\tau = 1$ in \eqref{KPPspeed2Dunsteadyflow} and $\delta=0.5$ in the velocity field ${\bf v}$. We use the spectral method to obtain an accurate reference solution for the principal eigenvalue of \eqref{KPPspeed2Dunsteadyflow}.  For figure \ref{convergece-dt-2Dmix}, we choose $\lambda = 0.57$.  Figure \ref{convergece-dt-2Dmix} shows the convergence results of the Lagrangian method in computing the principal eigenvalue, where the convergence rate is $(\Delta t)^{1.31}$. Figure \ref{mulambdaoverlambda-2Dmix} shows the 
numerical results of $\frac{\mu(\lambda)}{\lambda}$ for different $\lambda$'s, from which we can compute 
the KPP front speed in the 2D unsteady cellular flow. We can 
see that $\frac{\mu(\lambda)}{\lambda}$ is convex within the computational domain of $\lambda$. Thus,  we can compute the KPP front speed by finding the minimizer of $\frac{\mu(\lambda)}{\lambda}$.

\begin{figure}[tbph]
	\centering
	\begin{subfigure}{0.45\textwidth}
		\includegraphics[width =  \linewidth]{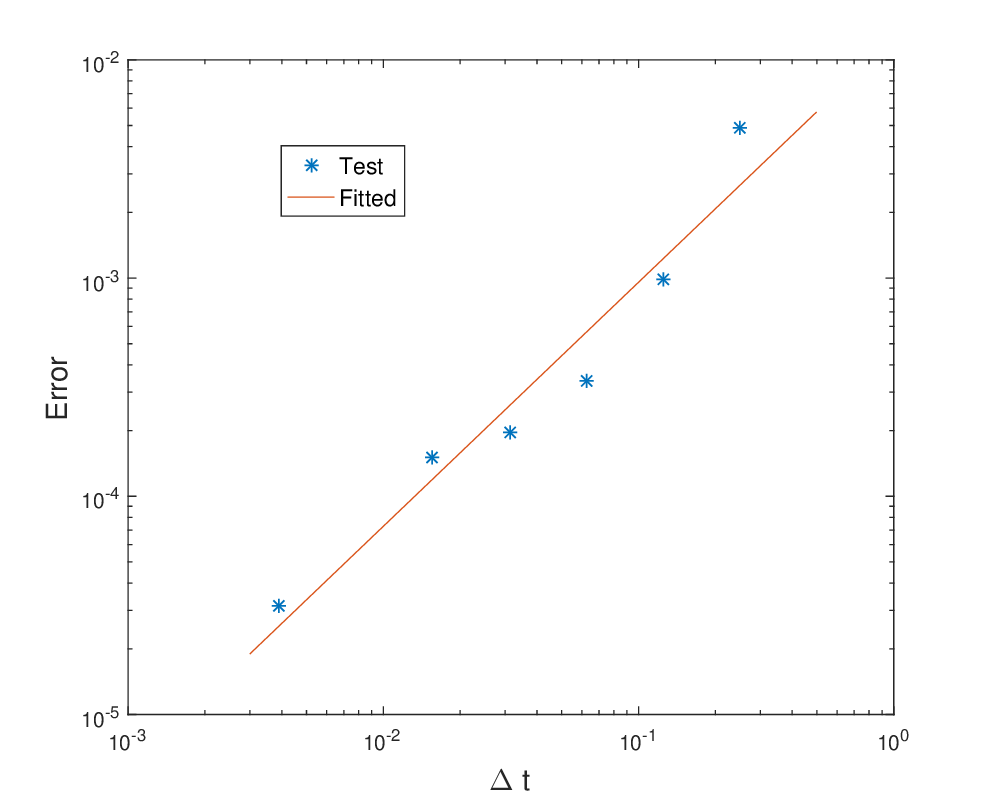}
		\caption{Convergence test for different $\Delta t$'s. The fitted slope is $\approx1.31$.}
		\label{convergece-dt-2Dmix}
	\end{subfigure}
	\begin{subfigure}{0.45\textwidth}
		\includegraphics[width =  \linewidth]{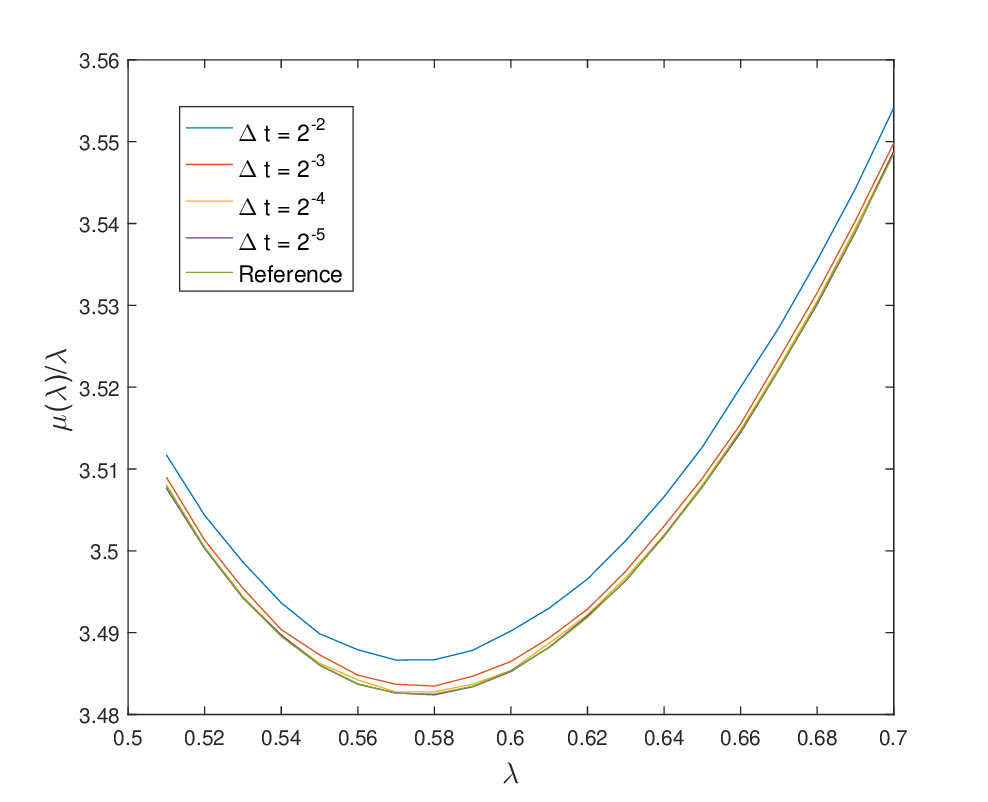}
		\caption{Numerical results of $\frac{\mu({\lambda})}{{\lambda}}$ for different $\lambda$'s.}
		\label{mulambdaoverlambda-2Dmix}
	\end{subfigure}
	\caption{Numerical results for a 2D unsteady cellular flow. }
	\label{fig:numericalresult2Dmixflow}
\end{figure}

\subsection{Investigate the dependence of front speed on the strength of the flows}\label{sec:NumCompKPPspead}
\noindent
To further test the performance of the Lagrangian method, we study the dependence of the KPP front speeds 
on the strength of different flows. Moreover, we study the relationship between the KPP front speeds in the chaotic flows and the effective diffusivity of the passive tracer model in the same chaotic flows. We refer the interested reader to \cite{WangXinZhang:18,Zhongjian2018sharp,lyu2019convergence} for the recent development in computing effective diffusivities in chaotic and random flows. We set the diffusion constant $\kappa=1$ and the time scale of reaction rate $\tau=1$.

Let us first consider this issue in KPP front speeds of time-independent flows. If we scale ${\bf v}\rightarrow A{\bf v}$, Eq.\eqref{KPPspeed2Dsteadyflow} can be rewritten as the following form 
\begin{equation}
\begin{aligned}
&  \Delta_{\textbf{x}} \Phi + (2  \lambda {\bf e }+A{\bf v})\cdot \nabla_{\textbf{x}} \Phi + \big(  {\lambda}^2 + {\lambda}{\bf e}\cdot A{\bf v} +  f'(0) \big)\Phi = \mu(\lambda) \Phi. 
\end{aligned}
\label{KPPspeed2DsteadyflowLargeA}
\end{equation}
The KPP front speed is $c^* = \inf_{{\lambda} > 0} \frac{\mu(\lambda)}{{\lambda}}$. Notice that the KPP front speed $c^*$ depends on $A$, i.e., $c^* = c^*(A)$.  
Therefore, we consider the equivalent equation
\begin{equation}
\begin{aligned}
& A^{-1}\Delta_{\textbf{x}} \Phi + (2  A^{-1}\lambda {\bf e }+{\bf v})\cdot \nabla_{\textbf{x}} \Phi + \big(  A^{-1}{\lambda}^2 + {\lambda}{\bf e}\cdot {\bf v} +  A^{-1}f'(0) \big)\Phi = \widetilde\mu(\lambda) \Phi,
\end{aligned}
\label{KPPspeed2DsteadyflowLargeArescale}
\end{equation}
where $\widetilde\mu(\lambda)=A^{-1}\mu(\lambda)$. Let $\widetilde c^*$ denote the KPP front speed of the rescaled equation \eqref{KPPspeed2DsteadyflowLargeArescale}. We have that 
\begin{equation} 
\widetilde c^* = \inf_{\lambda > 0} \frac{\widetilde\mu(\lambda)}{\lambda} = \frac{c^*}{A}.
\label{cstar}
\end{equation}
We denote $\sigma = A^{-1}$. For the 2D steady cellular flow ${\bf v} = (-\sin x_1\cos x_2,\cos x_1\sin x_2)$, it has been proved that $c^*(A) = O(A^{1/4})$ \cite{audoly2000reaction,novikov2007boundary}. 
Let $D^{E}(A)$ denote the effective diffusivity corresponding to the passive tracer model in the same 2D steady cellular flow ${\bf v}$. It has been proved by a boundary layer analysis that $D^{E}(A)=O(A^{1/2})$ in 
\cite{audoly2000reaction,childress1979alpha}. \textcolor{black}{By scaling analysis, we obtain that for the 2D steady cellular flow the following result holds}
\begin{equation} 
 c^*(A)=O(\sqrt{D^{E}(A)}). \label{KPPspeed-EffectiveD} 
\end{equation}
\textcolor{black}{To the best of our knowledge, the above relationship between the KPP front speeds and the effective diffusivity was only proved in 2D steady cellular flows; see \cite{novikov2007boundary,ryzhik2007kpp}.
The result \eqref{KPPspeed-EffectiveD} implies that $\widetilde c^*(\sigma) = \sigma O(\sigma ^{-1/4}) = O(\sigma ^{3/4})$, which provides a theoretical guidence for our numerical experiments. Figure \ref{cstaroversigma-2Dcellular} shows the numerical results of $\widetilde c^*(\sigma)$ in the 2D steady cellular flow obtained by our method. From the numerical results, we compute regression and obtain $\widetilde c^*(\sigma) = O(\sigma ^{0.74})$, 
which agrees with the theoretical result \eqref{KPPspeed-EffectiveD}.} 

\textcolor{black}{For other flows, such as unsteady flows and 3D chaotic flows, the understanding of $c^*(A)$ for large $A$'s (or  $\widetilde c^*(\sigma)$ for small $\sigma$'s) remains open. We will study these flows here.  
In our previous work \cite{Zhongjian2018sharp}, we computed the effective diffusivity of the passive tracer model in the 3D Kolmogorov flow, where ${\bf v} = (\sin x_1, \sin x_2,\sin x_3)$, and obtained that $D^{E}(A)=O(A^{1.13})$. Notice that in \cite{Zhongjian2018sharp} the effective diffusivity is represented in terms of the diffusion and we have converted the result in terms of the strength of the flows here, which are equivalent. The result \eqref{KPPspeed-EffectiveD} implies that $\widetilde c^*(\sigma) = \sigma O(\sigma ^{-0.56}) = O(\sigma ^{0.44})$. Using our method, we compute $\widetilde c^*(\sigma)$ for $\sigma$ in 3D Kolmogorov flow and show the numerical results in Figure \ref{cstaroversigma-3Dkpp}. We obtain that $\widetilde c^*(\sigma) = O(\sigma ^{0.43})$, which means that the result \eqref{KPPspeed-EffectiveD} also holds in the 3D Kolmogorov flow. We conjecture that the result \eqref{KPPspeed-EffectiveD} also holds true in other 3D chaotic flows. We will study this issue in future works. }  


\begin{figure}[tbph]
	\centering
	\begin{subfigure}{0.45\textwidth} \label{2dsteady}
		\includegraphics[width = \linewidth]{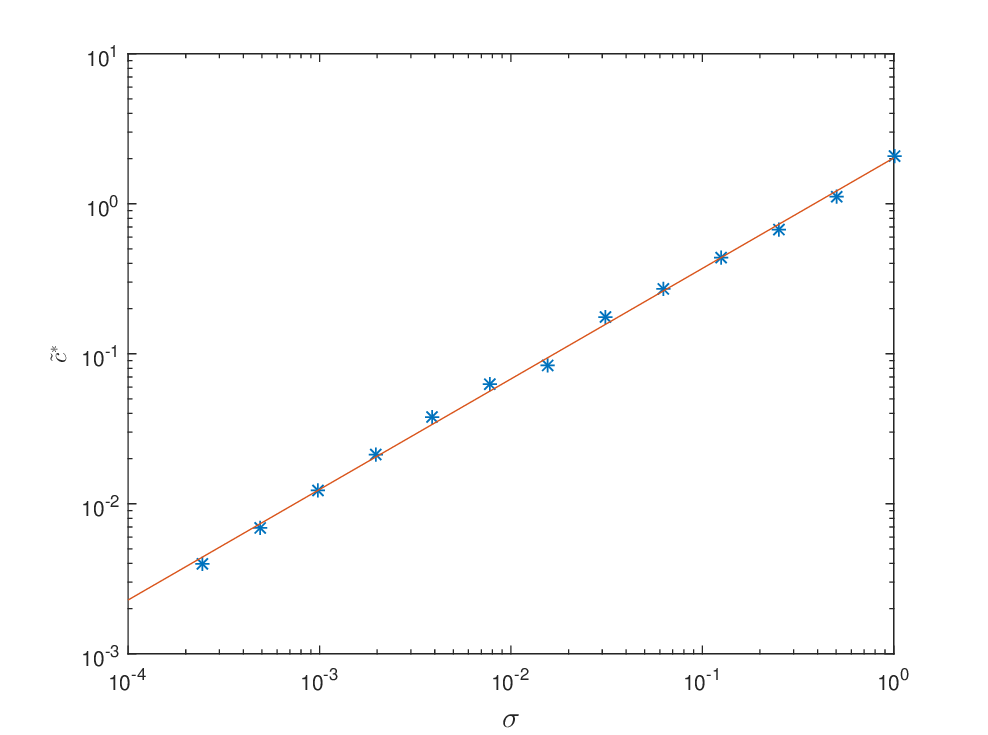}
		\caption{Numerical results of $\widetilde c^*(\sigma)$ in 2D cellular flow. The fitted slope is $\approx 0.74$.}
		\label{cstaroversigma-2Dcellular}
	\end{subfigure}
	\begin{subfigure}{0.45\textwidth} \label{3dksteady}
		\includegraphics[width = \linewidth]{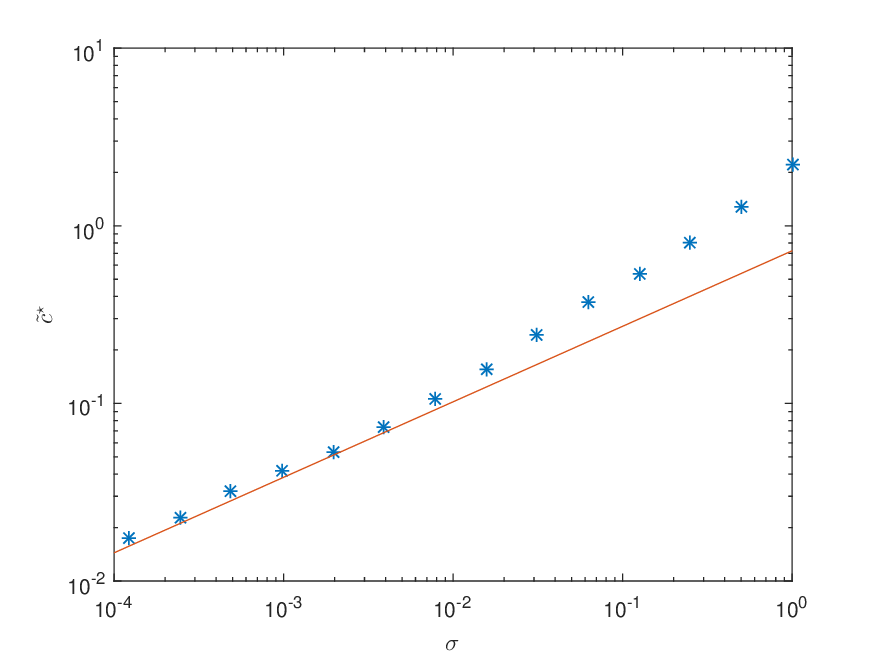}
		\caption{Numerical results of $\widetilde c^*(\sigma)$ in 3D Kolmogorov flow. The fitted slope is $\approx 0.43$.}
		\label{cstaroversigma-3Dkpp}
	\end{subfigure}
	\caption{Numerical results of $\widetilde c^*(\sigma)$ in different flows.}
	\label{fig:cstaroversigma-steadyflows}
\end{figure}

Next, we study the dependence of the KPP front speeds on the strength of time-dependent flows.
Specifically, we will consider two 3D flows. 
The first one is a time-dependent Kolmogorov flow with  ${\bf v} = \big(\sin(x_3+\theta \sin(2\pi t)),\sin(x_1+\theta \sin(2\pi t)), \sin(x_2+\theta \sin(2\pi t))\big)$, and the second one is a time-dependent ABC flow with ${\bf v} = \big(\sin(x_3+ \sin(2\pi \Omega t))+ \cos(x_2+ \sin(2\pi \Omega t)),\sin(x_1+\sin(2\pi \Omega t))+\cos(x_3+\sin(2\pi \Omega t)) , \sin(x_2+\sin(2\pi \Omega t))+\cos(x_1+\sin(2\pi \Omega t))\big)$.


\textcolor{black}{For the 3D time-dependent Kolmogorov flow, we choose iteration time $n=256$, time step $\Delta t=2^{-9}$ and particle number $N=400,000$.  Figure \ref{cstaroversigma-3DTDKflow} shows the result of $\widetilde c^*(\sigma)$ for small $\sigma$'s and different $\theta$'s. Again, we find the KPP front speed $\widetilde c^*(\sigma)$ is not very sensitive to the paramater $\theta$. When $\theta=1$, we obtain that $\widetilde c^*(\sigma) = O(\sigma ^{0.39})$.}




\begin{figure}[h]
	\centering
	\includegraphics[width=0.5\linewidth]{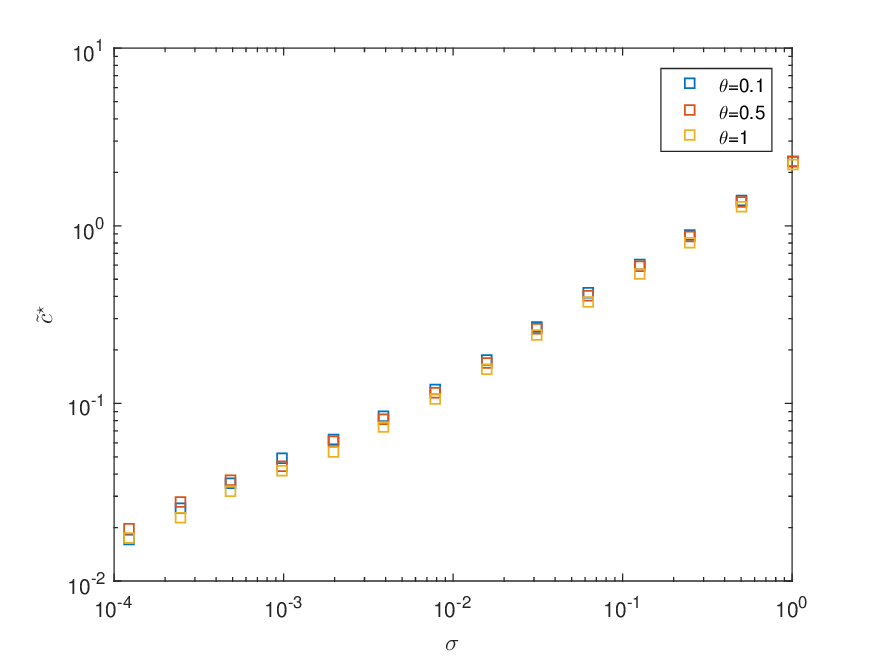}
	\caption{Numerical results of $\widetilde c^*(\sigma)$ in a 3D time-dependent Kolmogorov flow.}
	\label{cstaroversigma-3DTDKflow}
\end{figure}




In Figure \ref{fig:3dkfitting}, we plot out procedure searching for the $\lambda$ when the minimal in Eq.\eqref{cstar} was reached. We use $a\lambda + b\lambda^{-1} + c$ to fit a curve, then find the minimum of the curve. When $\sigma$ is large, the relative fluctuation is small and the minimum is easily to be found. When $\sigma$ is small, the relative fluctuation becomes strong enough, so we decide to fit the curve, then find the minimum point.

\begin{figure}[h]
	\centering
	\begin{subfigure}{0.30\textwidth}
		\includegraphics[width=\linewidth]{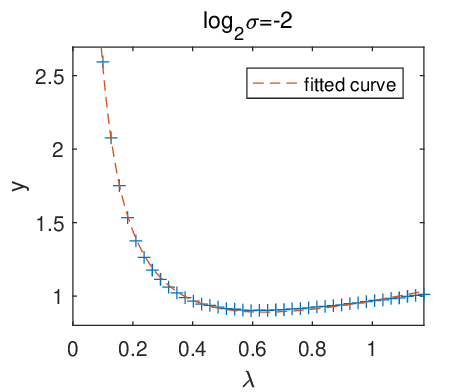}
	\end{subfigure}
	\begin{subfigure}{0.30\textwidth}
		\includegraphics[width=\linewidth]{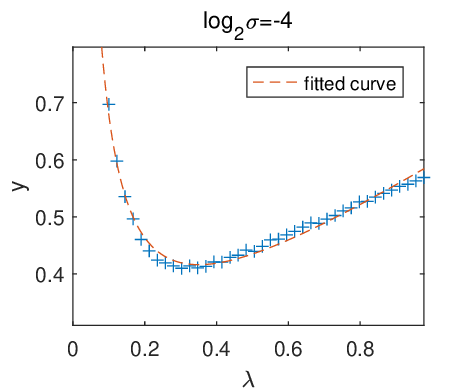}
	\end{subfigure}
	\begin{subfigure}{0.30\textwidth}
		\includegraphics[width=\linewidth]{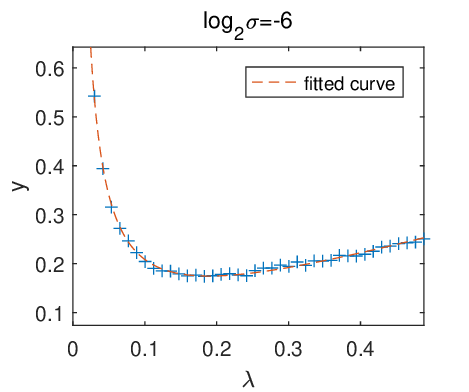}
	\end{subfigure}\\
	\begin{subfigure}{0.30\textwidth}
		\includegraphics[width=\linewidth]{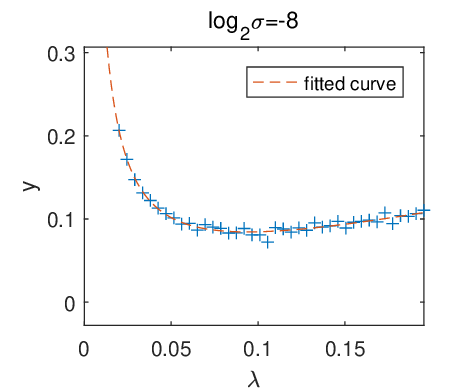}
	\end{subfigure}
	\begin{subfigure}{0.30\textwidth}
		\includegraphics[width=\linewidth]{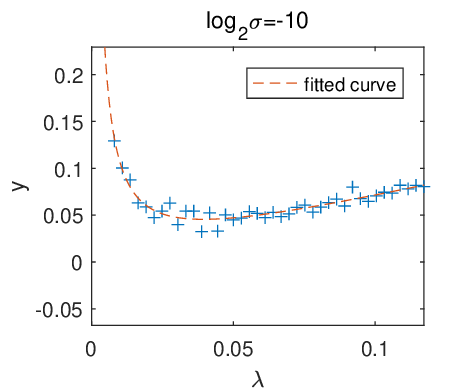}
	\end{subfigure}
	\begin{subfigure}{0.30\textwidth}
		\includegraphics[width=\linewidth]{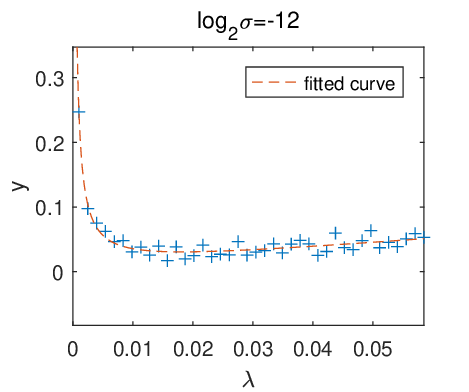}
	\end{subfigure}
	\caption{Numerical results of $\frac{\mu({\lambda})}{{\lambda}}$ for different $\lambda$'s and $\sigma$'s in the 3D time-dependent Kolmogorov flow. The red dash curve is fitted by $a\lambda+b\lambda^{-1}+c$.}
	\label{fig:3dkfitting}
\end{figure}

\textcolor{black}{For the 3D time-dependent ABC flow, we choose the iteration time $n=2048$ (since the ABC flow is more chaotic), time step $\Delta t=2^{-9}$, and particle number $N=400,000$. Figure \ref{fig:cstaroversigma3DTDABCflow} shows the KPP front speeds $\widetilde c^*(\sigma)$ for different $\Omega$'s, where $\Omega$ ranges from $2^{-7}$ to $2^0$. 
Figure \ref{fig:ABCLyapunovExponent} shows the slope of each approximation line for each $\Omega$ in Figure \ref{fig:cstaroversigma3DTDABCflow}. If we assume $\widetilde c^*(\sigma)=O(\sigma ^{\alpha})$ is true, the slope values in  Figure \ref{fig:ABCLyapunovExponent} give the power value $\alpha$'s for different $\Omega$'s.  We find that when $\Omega$ is near $0.1$, the power value $\alpha$ is large. When $\Omega$ is away from $0.1$, say $\Omega<2^{-4}$ or $\Omega>2^{-2}$, the power value $\alpha$ is small. A similar sensitive dependence on the frequency of time-dependent ABC flows was reported in \cite{brummell2001linear}, where the Lyapunov exponent of the deterministic time-dependent ABC flow problem (i.e., $\Omega=0$) was studied as the indicator of the extent of chaos; see Figure 2 and Figure 3 of \cite{brummell2001linear}.}



\begin{figure}[htbp]
	\centering
	\begin{subfigure}{0.45\textwidth} 
		\includegraphics[width = \linewidth]{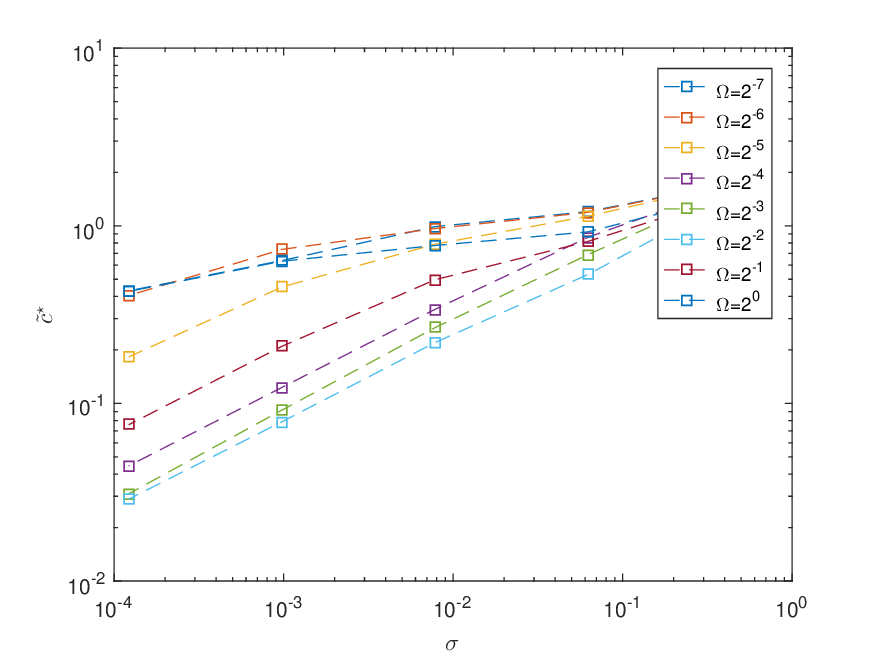}
		\caption{Numerical results of $\widetilde c^*(\sigma)$ for different $\Omega$'s and different $\sigma$'s.}
		\label{fig:cstaroversigma3DTDABCflow}
	\end{subfigure}
	\begin{subfigure}{0.45\textwidth} 
		\includegraphics[width = \linewidth]{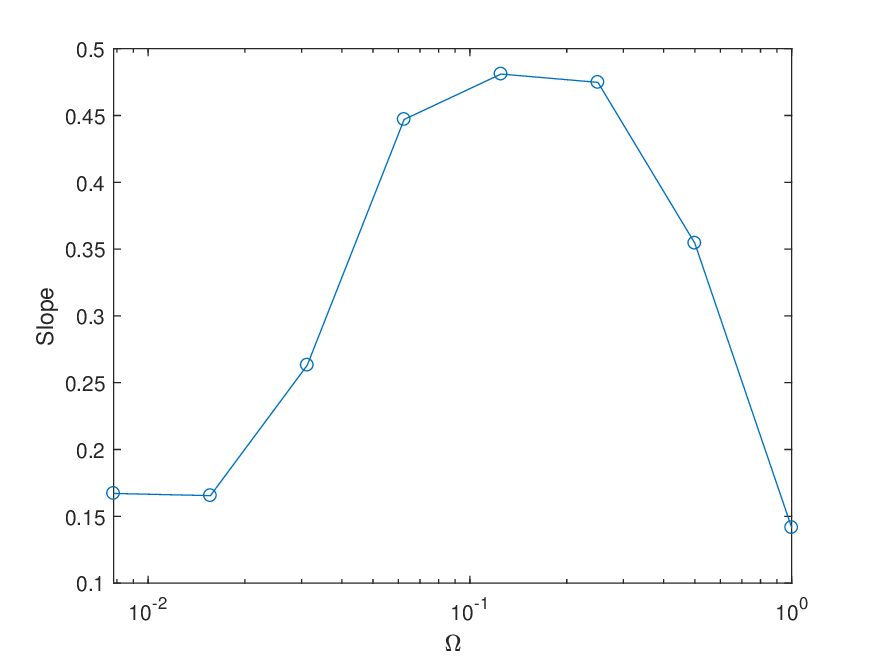}
		\caption{Power values of $\widetilde c^*(\sigma)=O(\sigma ^{\alpha})$ for different $\Omega$'s.}
		\label{fig:ABCLyapunovExponent}
	\end{subfigure}
	\caption{Numerical results for the time-dependent ABC flows.}
	\label{fig:TDABCdifferntparameters}
\end{figure}

\textcolor{black}{We compare the computational time of the interacting particle method and the spectral method in the 2D cellular flow example. The numerical experiments are carried out on the same core of the HPC2015 system at HKU with 10-core Intel Xeon E5-2600 v3(Haswell) processors and 96 GB physical memory. We compute the front speed using the spectral method mentioned in Section \ref{sec:NormConvergeTest}. We set the Fourier modes $H = 2^k$ and $k$ is a positive integer.  When $\sigma = 2^{-2}$, for the spectral method, $H = 2^3$ is enough and it spends $1.13$ seconds to calculate the front speed; while for our interacting particle method, the computational time is about $45.01$ seconds. When $\sigma = 2^{-5}$, for the spectral method, $H = 2^4$ is enough and it spends $42.35$ seconds to calculate the front speed, and the interacting particle method costs $172.76$ seconds. When $\sigma = 2^{-8}$, for the spectral method, $H = 2^4$ is needed and it costs $1203.12$ seconds to calculate the front speed; on the other hand, our interacting particle method costs $676.23$ seconds. When $\sigma$ becomes extremely small, the spectral method becomes very expensive, however, our interacting particle method is still very efficient. For instance, when $\sigma = 2^{-12}$, the spectral method may need several days to calculate the front speed, but our interacting particle method only costs $5378.24$ seconds. We remark that the spectral method becomes very expensive in computing front speeds for 3D chaotic flows. However, the computational time of the interacting particle method only weakly depends on the dimension of the physical space. Thus, we can compute KPP front speeds in 3D chaotic flows.}

\subsection{Evolution of the empirical distribution of the particles}\label{sec:EvolutionParticles}
\noindent
As stated in Theorem \ref{thm:existenceinvariantmeasure}, the empirical distribution converges to the 
invariant measure of Feynman-Kac semigroup as $n$ approaches infinity. Our Lagrangian method can not only calculate the principal eigenvalue but also compute the evolution of the distribution. In this subsection, we study the empirical distribution of the $N$-IPS system moduled to the torus space $\mathbb{T}^d$. We choose the particle number $N=200,000$ in all the numerical experiments. 

 
Figure \ref{steady} shows the invariant distribution generated by the $N$-IPS system in the 2D steady cellular flow, where  
${\bf v} = (- \sin x_1 \cos x_2, \cos x_1 \sin x_2)$. The parameter $\sigma$ varies from $2^0$ to $2^{-5}$.  The strength of the convection is then proportion to $1/\sigma$. We can see that when we increase the strength, 
the invariant measure concentrates in smaller domains and its gradient becomes sharper near these domains, which is a common phenomenon in fluid dynamics. In addition, by comparing to the pattern at the boundary of the plot, one can find that the invariant measure is periodic in physical space. 

\begin{figure}[htbp]
	\centering
	\includegraphics[width=0.9\linewidth]{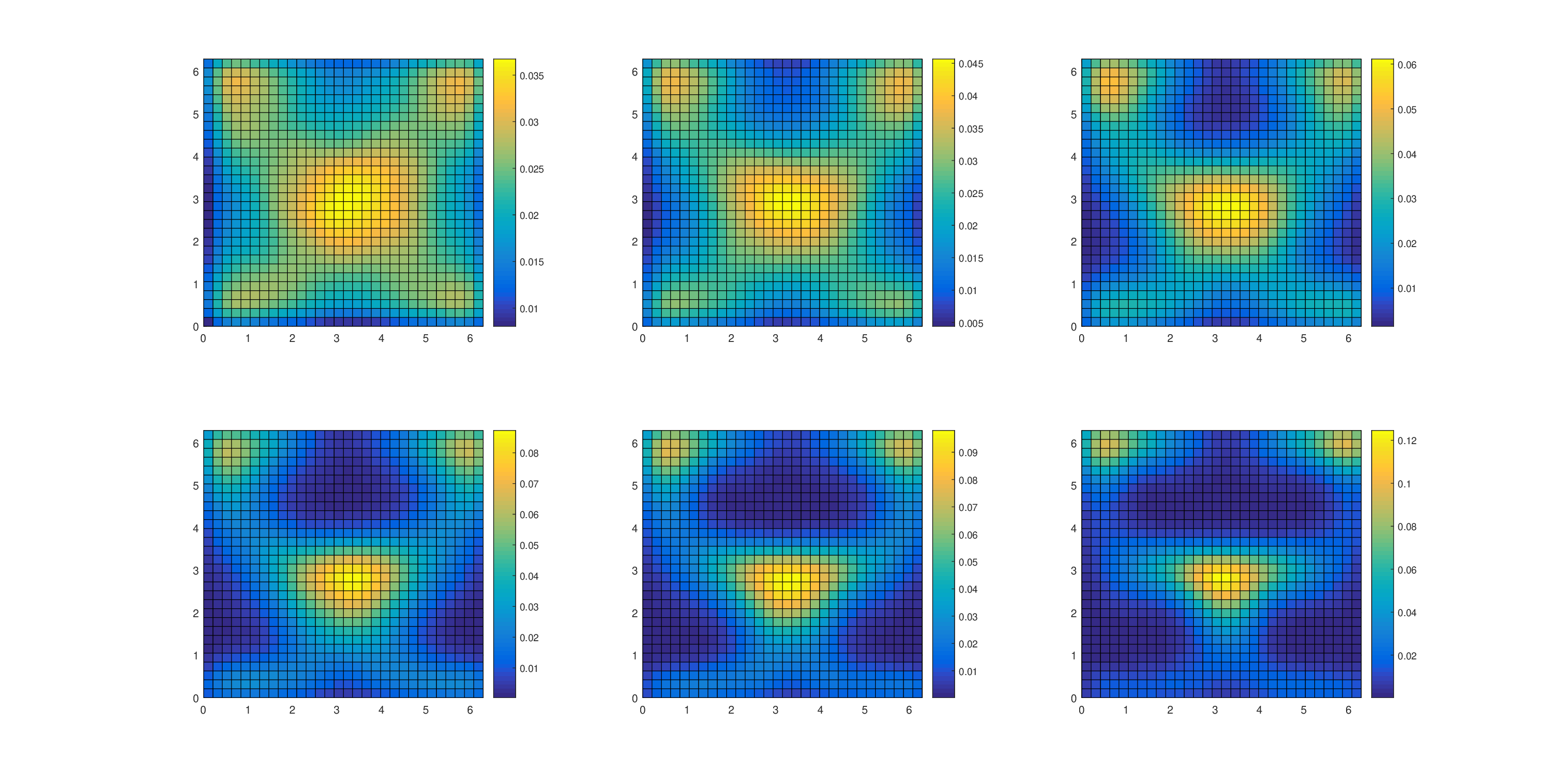}
	\caption{Empirical distributions for the 2D steady cellular flow with $\sigma$ varies from $2^0$ to $2 ^{-5}$. First row from left to right: $\sigma=2^0$, $\sigma=2^{-1}$, and $\sigma=2^{-2}$. Second row from left to right: $\sigma=2^{-3}$, $\sigma=2^{-4}$, and $\sigma=2^{-5}$.}
	\label{steady}
\end{figure} 

Next, we study the evolution of invariant distribution generated by our $N$-IPS system in a 2D time-periodic mixing flow, where ${\bf v} = (-\cos x_2-\theta \sin x_1\cos(2\pi t),\cos x_1+\theta \sin x_2\cos(2\pi t))$. Figure \ref{periodicchange} shows the empirical distribution of the  $N$-IPS system at different times within one period when the iteration time $n=400$. From these numerical results, we can see the invariant distribution varies at different times within one period. And the first subfigure and last subfigure are identical. These results are consistent with our analysis obtained in Lemma \ref{eig-estimate}, where we proved that the invariant measure changes periodically with the same period as the flow. 

\begin{figure}[htbp]
	\centering
	\includegraphics[width = 0.9\linewidth]{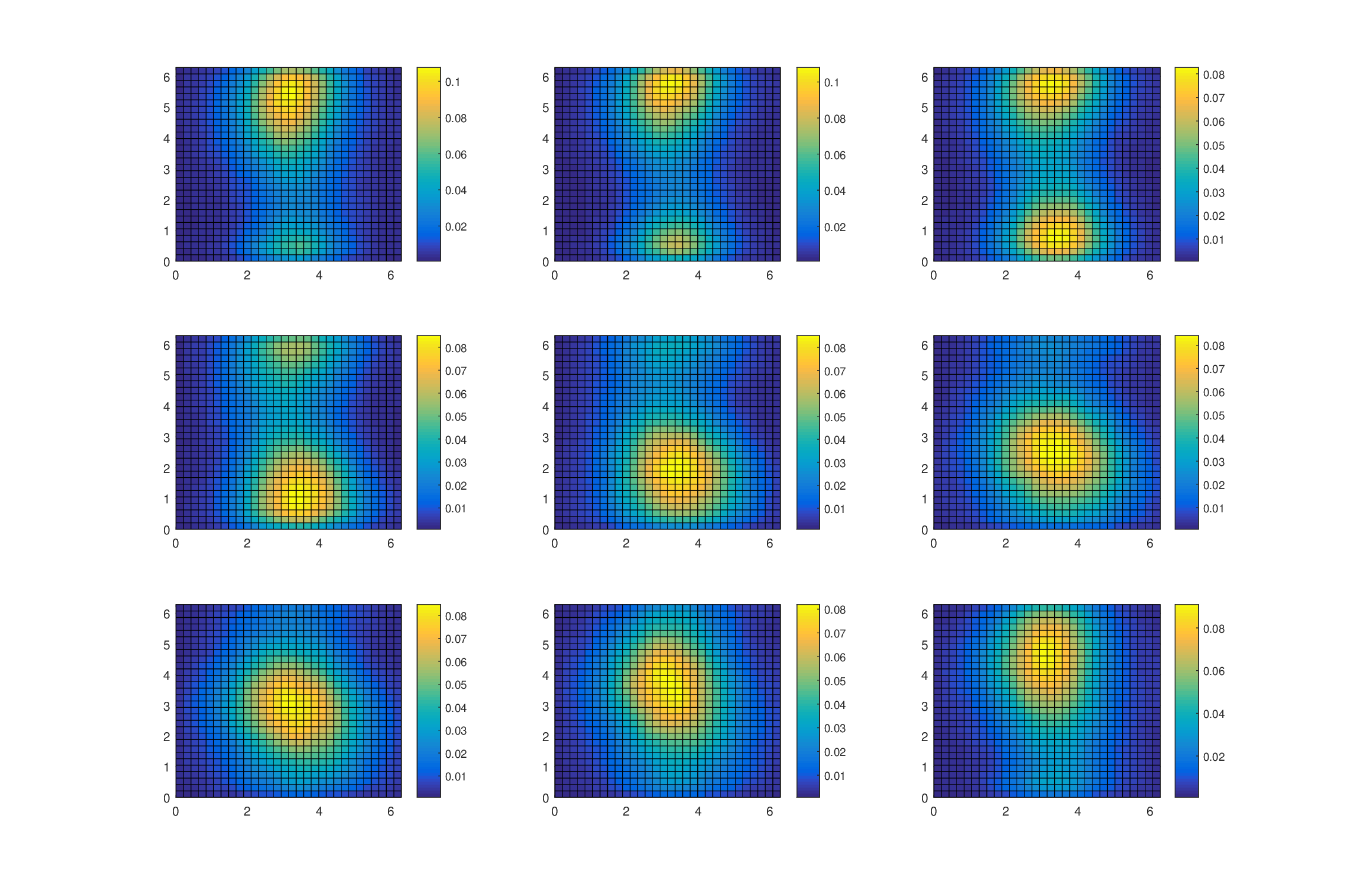}
	\caption{Empirical distributions for the 2D time-periodic mixing flow with $\theta = 1$, $\sigma = 1$, in different phase of one period: $t$ varies from $0$ to $1$ with time interval equal to $1/9$.}
	\label{periodicchange}
\end{figure}

Finally, we let the parameter $\sigma$ vary from $2^0$ to $2^{-5}$ and study the evolution of invariant distribution generated by our $N$-IPS system in the 2D time-periodic mixing flow. Figure \ref{nonsteady} shows that with the increasing of the strength of the convection, the invariant measure becomes compactly supported with a sharp gradient. 

From these numerical results, we get two conclusions. First, the invariant measure of the Feynman-Kac semigroup associated with the KPP operator is no longer uniform distribution. This is due to the effect from the potential function $c(t,\textbf{x})$. Second, the invariant measure converges to a limiting measure as $\sigma\to  0$. Notice that when $\sigma$ is small, the invariant measure develops sharp gradients, which requires more particles to compute. Moreover, it may take more iteration time steps to converges.  Developing effective sampling methods to compute the invariant measure for the KPP operator with small diffusion constant will be studied in our future works.




\begin{figure}[htbp]
	\centering
	\includegraphics[width = 0.9\linewidth]{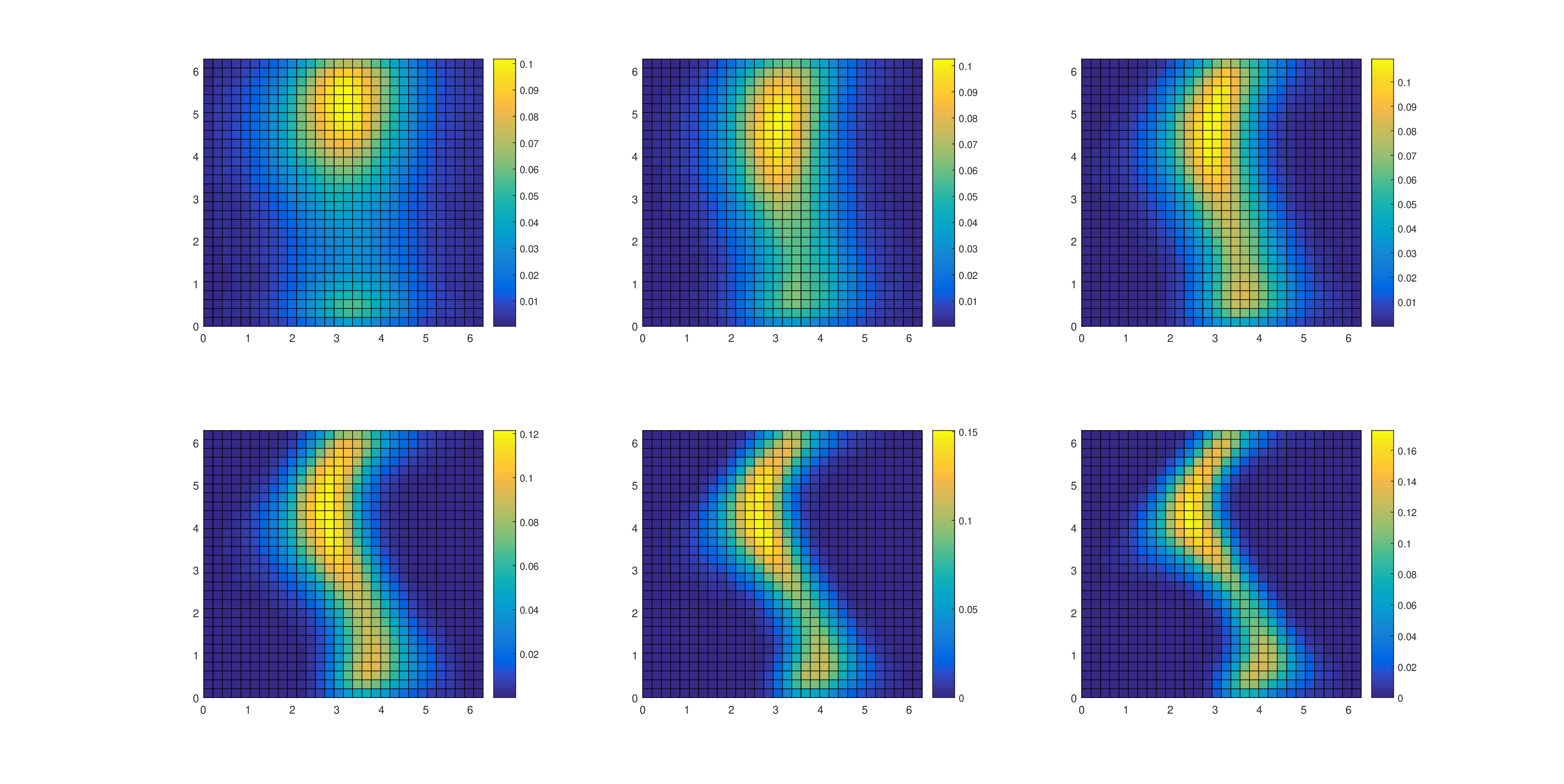}
	\caption{Empirical distributions for the 2D time-periodic mixing flow with $\theta  = 1$, $\sigma$ varies from $2^0$ to $2^{-5}$. First row from left to right: $\sigma=2^0$, $\sigma=2^{-1}$, and $\sigma=2^{-2}$. Second row from left to right: $\sigma=2^{-3}$, $\sigma=2^{-4}$, and $\sigma=2^{-5}$.}
	\label{nonsteady}
\end{figure}

\section{Conclusion}\label{sec:Conclusion}
\noindent 
In this paper, we developed efficient Lagrangian particle methods to compute the KPP front speeds in time-periodic cellular and chaotic flows and provided rigorous convergence analysis for the numerical schemes. In the convergence analysis, we first obtained the error of the operator splitting methods in approximating the solution operator corresponding to the linearized KPP equation. Then, we proved the convergence of the Lagrangian particle method in computing the principal eigenvalue based on the Feynman-Kac semigroup theory. Finally, we presented numerical results to verify the convergence rate of the proposed method for computing the principal eigenvalues. In addition, we computed the KPP front speeds in several typical chaotic flow problems of physical interests, including the Arnold-Beltrami-Childress (ABC) flow and the Kolmogorov flow. It has been proved that the KPP front speed and the effective diffusivity satisfies the relation $c^*(A)=O(\sqrt{D^{E}(A)})$ in 2D cellular flows \cite{novikov2007boundary,ryzhik2007kpp}. We numerically verified this relation and found that this relation
still holds in 3D Kolmogorov flows and ABC flows.

There are three directions we plan to explore in our future work. First, we will extend the
Lagrangian particle method to compute  KPP front speeds in time-stochastic and space-periodic flows. Second, 
we will develop Lagrangian particle methods to compute KPP fronts speeds in more complex fluid flows, where the computational domain is not compact. This type of problem is more challenging both analytically and numerically. As stated in the introduction part, there is limited literature on studying the existence of KPP front speeds in complex flows. In the aspect of numerical computation, our current method cannot be adapted to non-compact domains.  We shall adopt some relaxation techniques to address this problem.  In addition, we shall develop adaptive sampling methods for our Lagrangian particle methods in order to resolve the sharp gradients in the invariant measure when the magnitude of the velocity field is very large.

\section*{Acknowledgement}
\noindent
The research of J. Lyu and Z. Wang is partially supported by the Hong Kong Ph.D. Fellowship Scheme. 
The research of J. Xin is partially supported by NSF grants DMS-1924548 and DMS-1952644. The research of Z. Zhang is supported by Hong Kong RGC grants (Projects 17300817 and 17300318), Seed Funding Programme for Basic Research (HKU), and Basic Research Programme (JCYJ20180307151603959) of The Science, Technology, and Innovation Commission of Shenzhen Municipality. The computations were performed using research computing facilities offered by Information Technology Services, the University of Hong Kong.


\appendix

\section{Error bounds for exponential operator splitting in non-autonomous evolution equations}
\label{Sec:SemigroupTheory}   
\subsection{Euler methods for non-autonomous evolution equations}\label{Sec:Euler}
\noindent
In this section, we review the fundamental rsults for abstract linear evolution equations by semigroup theory; see e.g. \cite{engel1999one,chicone1999evolution} for more details. We consider the non-autonomous Cauchy problem (NCP) as follows
\begin{equation}\label{NCPE-appendix}
\left\{
\begin{aligned}
\frac{d}{dt} u(t) &= \mathcal{A}(t) u(t), \quad t\ge s\in \mathbb{R}\\
u(s) &= x \in X,
\end{aligned}
\right.
\end{equation}
where $X$ is a Banach space and  $(\mathcal{A}(t), \mathcal{D}(\mathcal{A}(t)))_{t\in \mathbb{R}}$ is a family of linear operators on $X$.

\begin{dfn}\label{NCPD}
	A continuous function $u: [s,\infty) \longrightarrow X$ is called a classical solution of \eqref{NCPE-appendix} if $u \in C^1([s,\infty); X), u(t) \in \mathcal{D}(\mathcal{A}(t))$ for all $t\ge s$, $u(s) = x$, and $ \frac{d}{dt} u(t) = \mathcal{A}(t)u(t)$ for all $t\ge s$.
\end{dfn}

\begin{dfn}\label{NCPW}
	For a family $(\mathcal{A}(t), \mathcal{D}(\mathcal{A}(t)))_{t\in \mathbb{R}} $ of linear operators on a Banach space $X$, the NCP \eqref{NCPE-appendix} is well-posed with regularity subspace $(Y_s)_{s\in \mathbb{R}}$ and exponentially bounded solutions, if 
	\begin{enumerate}[(i)]
		\item (Existence) For all $s\in R$ the subspace 
		\begin{equation}
		Y_s = \{y\in X :\text{there exists a classical solution for the NCP \eqref{NCPE-appendix}}\}\subset \mathcal{D}(\mathcal{A}(s))
		\end{equation}
		is dense in X.
		\item (Uniqueness) For every $y\in Y_s$, the solution $u_s(\cdot,y)$ is unique.
		\item (Continuous dependence) The solution continuously depends  on $s$ and $y$, i.e., if $s_n \to s \in \mathbb{R}, ||y_n-y||_X \to  0 $ with $y_n \in Y_{y_n}$, then we have
		$ ||\hat u_{s_n} (t,y_n) - \hat u_{s}(t,y)||_X \to 0$
		uniformly for $t$ in compact subsets of $\mathbb{R}$, where
		$$ \hat u_{s}(t,y) = \left\{\begin{aligned} &u_r(t,y)  &if~ &r \le t,\\&y&if~&r > t.\end{aligned}\right.$$
		\item (Exponential boundedness) There exists a constant $\omega \in \mathbb{R}$ such that 
		$$||u_s(t,y)||_{X} \le e^{\omega(t-s)} ||y||_X$$ 
		for all $y \in Y_s$ and $t \ge s$.
	\end{enumerate}
\end{dfn}
\begin{dfn}
	A family $\{\mathcal{U}(t,s),t\ge s\}$ of linear, bounded solution operators on Banach space $X$ is called an exponentially bounded evolution family if
	\begin{enumerate}[(i)]
		\item $\mathcal{U}(t,r) \mathcal{U}(r,s) = \mathcal{U}(t,s)$ and $\mathcal{U}(t,t)=Id$ hold for all $t\ge r\ge s \in \mathbb{R}$,
		\item the mapping $(t,s)\to \mathcal{U}(t,s)$ is strongly continuous,
		\item $||\mathcal{U}(t,s)||_X \le e^{\omega (t-s)}$ for some $ \omega\in\mathbb{R}$ and all $t\ge s \in \mathbb{R}$.
	\end{enumerate}
\end{dfn}

In contrast to the behavior of $\text{C}_0$-semigroups, the algebraic proposition of an evolution family do not imply any differentiability on a dense subspace. Therefore, we need extra assumptions in order to solve an NCP.
\begin{dfn}
	An evolution family $\{\mathcal{U}(t,s),t\ge s\}$ is called evolution family solving NCP \eqref{NCPE-appendix} if for every $s\in \mathbb{R}$ the regularity space
	$$Y_s = \{y\in X: [s,\infty) \ni t\mapsto \mathcal{U}(t,s)y\text{ solves NCP \eqref{NCPE-appendix}}\}$$
	is dense in X.
\end{dfn}
In this case, the unique classical solution of the NCP \eqref{NCPE-appendix} is given by $u(t) = \mathcal{U}(t,s)x$. The well-posedness of the NCP \eqref{NCPE-appendix} can now be characterized by the existence of solving an evolution family $\{\mathcal{U}(t,s),t\ge s\}$.

\begin{pro}
	Let $X$ be a Banach space and $(\mathcal{A}(t),\mathcal{D}(\mathcal{A}(t)))_{t \in \mathbb{R}}$ be a family of linear operators on $X$. The following assertions are equivalent \cite{engel1999one}.
	\begin{enumerate}[(i)]
		\item The NCP \eqref{NCPE-appendix} is well-posed.
		\item There exits a unique evolution family $\{\mathcal{U}(t,s),t\ge s\}$ solving the NCP \eqref{NCPE-appendix}.
	\end{enumerate}
	In addition, if $||e^{\tau \mathcal{A}(t)}||_X \le e^{\omega \tau}$ for any $\tau \ge 0, t\in \mathbb{R}$, then we have $||\mathcal{U}(t,s)||_X \le e^{\omega (t-s)}$.
\end{pro}

The well-posedness of non-autonomous evolution equations is complicated and there is no general theory describing it. Conditions implying well-posedness are generally divided into parabolic-type assumptions and hyperbolic-type ones. Due to the property of the KPP equation, 
we only study the parabolic-type conditions in this paper, where the domain $(\mathcal{D}(\mathcal{A}(t))$ is independent of $t \in \mathbb{R}$. We refer the interested reader to \cite{schnaubelt2002well} for  more general cases.

\begin{asm}\label{parabolic} (Parabolic-type conditions)
	\begin{enumerate}[(P1)]
		\item The domain $\mathcal{D} = \mathcal{D}(\mathcal{A}(t))$ is independent of $t\in \mathbb{R}$. 
		\item For each $t\in R$ the operator $\mathcal{A}(t)$ is sectorial and generates an analytic semigroup $e^{\cdot \mathcal{A}(t)}$. For all $t\in \mathbb{R}$, the resolvent $\mathcal{R}({\gamma_1}, \mathcal{A}(t))$ exists for all ${\gamma_1} \in \mathbb{C}$ with $\text{Real}{\gamma_1} \ge 0$ and there is a constant $M \ge 1$ such that
		\begin{equation}
		\big|\big|R({\gamma_1}, \mathcal{A}(t))\big|\big|_X \le \frac{M}{|{\gamma_1}| + 1}
		\end{equation} 
		for $\text{Real}{\gamma_1} \ge 0$ and $t\in \mathbb{R}$. The semigroups $e^{\cdot \mathcal{A}(t)}$ satisfy $||e^{\tau \mathcal{A}(t)}||_X \le e^{\omega \tau}$ for some constant $\omega \in \mathbb{R}$.
		\item There exist constants $L\ge 0$ and $0< \theta \le 1$ such that
		\begin{equation}
		\big|\big|(\mathcal{A}(t)-\mathcal{A}(s))\mathcal{A}(0)^{-1} \big|\big|_X \le L|t-s|^\theta,\text{  for all }t,s\in \mathbb{R}.
		\end{equation} 
	\end{enumerate}
\end{asm}
To obtain a convergence estimate for the operator in certain norm,  we need an additional assumption on $\mathcal{A}(t)$ as follows. 
\begin{asm}\label{continuous-appendix}
	The operator $\mathcal{A}(t)$ satisfies a H\"{o}lder continuous condition. Namely, there exists $0 \le \alpha < \beta $ such that for any $x \in \mathcal{D}(\mathcal{A})$, 
	\begin{equation}
	\big|\big|(\mathcal{A}(t)-\mathcal{A}(s))x\big|\big|_X \le C|t-s|^\beta||\mathcal{A}(\tau) x||^{\alpha}_X ||x||_X^{1-\alpha},
	\end{equation}  
	for any $s\le\tau\le t$.
\end{asm}

For forward Euler type discretization, Assumption \ref{continuous-appendix} can be relaxed to $\tau = s$ only. The backword Euler type discretization needs $\tau = t$, and other discretization methods  need different $\tau$'s instead. For analytic semigroups, the following estimate holds true \cite{engel1999one,pazy2012semigroups}.
\begin{lem}\label{Analytic}
	Let $e^{t\mathcal{A}}$ be an anlytical semigroup on $X$. Let $\mathcal{A}$ be the infinitesimal generator. There is a constant $C \ge 0$ such that 
	\begin{equation}
	||\mathcal{A} e^{t \mathcal{A}}||_X \le \frac{C}{t},  \quad t>0, \quad 0\le\alpha \le 1.
	\end{equation}
\end{lem}

Now we state the first result, which gives the approximation error of the freezing time coefficients
methods for solving the NCP \eqref{NCPE-appendix}. 

\begin{thm}\label{Euler-appendix} 
	Suppose assumptions \ref{parabolic} and \ref{continuous-appendix} hold true.  Let $\mathcal{U}(T,0)$ be the solution operator associated with the NCP \eqref{NCPE-appendix}. Then the solution operator obtained by the freezing time coefficients methods has the following approximation error to $\mathcal{U}(T,0)$ 
	\begin{equation}
	\big|\big|\mathcal{U}(T,0) - \prod_{k=0}^{M-1} e^{\Delta t \mathcal{A}(k\Delta t)}\big|\big|_X 
	\le C(T)(\Delta t)^{\beta-\alpha},
	\label{EulerApproximation0}
	\end{equation} 
	where $T > 0$, $M$ is an integer, and $\Delta t= \frac{T}{M}$.
\end{thm}

\begin{proof}
	First we refer to \cite{schnaubelt2002well} for the abstract version of the method of freezing coefficients, 
	\begin{equation}
	\mathcal{U}(t,s) = e^{(t-s)\mathcal{A}(s)} + \int_s^t \mathcal{U}(t,\tau) (\mathcal{A}(\tau) - \mathcal{A}(s))e^{(\tau-s)\mathcal{A}(s)}d\tau,
	\label{EulerApproximation1}
	\end{equation}	
	which immediately gives us that, for every $x \in X$,
	\begin{align}
	&\big|\big|(\mathcal{U}(t,s) - e^{(t-s)\mathcal{A}(s)})x\big|\big|_X \nonumber\\
	= &\big|\big|\int_s^t \mathcal{U}(t,\tau) (\mathcal{A}(\tau) -\mathcal{A}(s))e^{(\tau-s)\mathcal{A}(s)}xd\tau\big|\big|_X  \nonumber\\
	\le&\int_s^t \big|\big|\mathcal{U}(t,\tau)\big|\big|_X (\tau-s)^\beta \big|\big|\mathcal{A}(s) e^{(\tau-s)\mathcal{A}(s)}x\big|\big|_X^\alpha \big|\big|e^{(\tau-s)\mathcal{A}(s)}x\big|\big|^{1-\alpha}_X d\tau.
	\label{EulerApproximation2}
	\end{align}
	In \eqref{EulerApproximation2}, we have used the fact that $e^{(\tau-s)\mathcal{A}(s)}x \in \mathcal{D}(\mathcal{A}) $ for any $x \in X$. Notice that $\mathcal{A}(s)$ generates an analytic semigroup $e^{\cdot \mathcal{A}(s)}$. According to \eqref{Analytic}, we have the following estimate
	\begin{equation}
	\big|\big|\mathcal{A}(s) e^{(\tau -s )\mathcal{A}(s)}\big|\big|^\alpha_X \le C (\tau -  s)^{-\alpha} e^{\omega \alpha(\tau - s)}.
	\label{EulerApproximation3}
	\end{equation}
	Substituting \eqref{EulerApproximation3} into \eqref{EulerApproximation2}, we obtain that,
	\begin{align}
	&\big|\big|(\mathcal{U}(t,s) - e^{(t-s)\mathcal{A}(s)})x\big|\big|_X \nonumber\\ 
	\le &\int_s^t C e^{\omega (t-\tau)} (\tau-s)^{\beta-\alpha} e^{\omega (\tau - s)} d\tau ||x||_X = \frac{C}{1+\beta-\alpha}e^{\omega(t-s)} (t-s)^{1+\beta-\alpha} ||x||_X.
	\label{EulerApproximation4}
	\end{align}
	Thus, we get the estimate for the operator in the norm $||\cdot||_X$ 
	\begin{equation}
	\big|\big|\mathcal{U}(t,s) - e^{(t-s)\mathcal{A}(s)}\big|\big|_X \le \frac{C}{1+\beta-\alpha}e^{\omega(t-s)} (t-s)^{1+\beta-\alpha}.
	\label{EulerApproximation5}
	\end{equation}
	We denote $\mathcal{U}(T,0) = \prod_{k=0}^{M-1} \mathcal{U}((k+1)\Delta t, k\Delta t)$. Using the telescoping sum argument, we obtain   
	\begin{align}
	&\big|\big|\mathcal{U}(T,0) - \prod_{k=0}^{M-1} e^{\Delta t \mathcal{A}(k\Delta t)}\big|\big|_X  \nonumber \\
	=& \Big|\Big|\sum_{j = 0}^{M-1} \prod_{k=j+1}^{M-1} U((k+1)\Delta t,k\Delta t) \big(\mathcal{U}((j+1)\Delta t,j\Delta t) - e^{\Delta t\mathcal{A}(j\Delta t)}\big)\prod_{l=0}^{j-1} e^{\Delta t\mathcal{A}(l\Delta t)}\Big|\Big|_X\nonumber\\
	\le& \sum_{j = 0}^{M-1} e^{\omega (N-j-1)\Delta t} \frac{C}{1+\beta-\alpha}e^{\omega \Delta t} (\Delta t)^{1+\beta-\alpha} e^{\omega j\Delta t} =  \frac{Ce^{\omega T}}{1+\beta-\alpha} (\Delta t)^{\beta-\alpha}.
	\end{align}
	The statement in \eqref{EulerApproximation0} is proved. 
\end{proof}

For higher order operator splitting methods,  in some specific situation the higher order convergence has been proved in \cite{hochbruck2003magnus,hochbruck2010exponential}. In their works, the assumption \ref{continuous-appendix} was largely strengthen, both for the operator $\mathcal{A}(t)$ and initial condition, and the convergence was largely depends on the graph norm $||v||_{\alpha} := ||\mathcal{A}(t)^\alpha v||_X$. 
The convergence in norm $||\cdot||_X$ is still open and will be our future research work.

\subsection{Operator splitting methods for solving non-autonomous evolution equations}
\noindent
We study the approximation error of operator splitting methods in solving non-autonomous evolution equations. 
To be specific, we consider an abstract NCP as follows
\begin{equation}\label{NCPE}
\left\{
\begin{aligned}
\frac{d}{dt} u(t) &= (\mathcal{A}(t)+\mathcal{B}(t)) u(t), \quad t\ge s\in \mathbb{R},\\
u(s) &= x \in X,
\end{aligned}
\right.
\end{equation}
on a Banach space $X$, where $\mathcal{A}(t)$ and $\mathcal{B}(t)$ are linear operators, $\mathcal{D}(\mathcal{A}(t))$ is independent of $t$ and dense in $X$, and for each $t \in \mathbb{R}$, $\mathcal{A}(t)$, $\mathcal{B}(t)$ and $\mathcal{A}(t) + \mathcal{B}(t)$ generate  strongly continuous semigroups $e^{\cdot \mathcal{A}(t)}$, $e^{\cdot \mathcal{B}(t)}$ and $ e^{\cdot(\mathcal{A}(t) + \mathcal{B}(t))}$, respectively. Furthermore, due to the property of evolution equation, solving $u(t)$ and solving $e^{{\gamma_1} t}u(t)$ is equivalent, we assume $\big|\big|e^{\tau \mathcal{A}(t)}\big|\big|_X \le 1$,$\big|\big|e^{\tau \mathcal{B}(t)}\big|\big|_X \le 1$,$\big|\big|e^{\tau (\mathcal{A}(t)+\mathcal{B}(t))}\big|\big|_X \le 1$.

We will study the NCP \eqref{NCPE} based on the perturbation theory. We assume $\mathcal{A}(t)$ is a sectorial operator, which generates an analytical semigroups $e^{\cdot \mathcal{A}(t)}$, and assume $\mathcal{B}(t)$ is bounded, thus $\mathcal{A}(t) + \mathcal{B}(t)$ is also sectorial and generates an analytical semigroups $e^{\cdot(\mathcal{A}(t) + \mathcal{B}(t))}$, where $\mathcal{D}(\mathcal{A}(t) + \mathcal{B}(t)) = \mathcal{D}(\mathcal{A}(t))$. In addition, we assume that the operator $\mathcal{A}(t) + \mathcal{B}(t)$ satisfies assumptions \ref{parabolic} and \ref{continuous-appendix}.
Therefore, the corresponding evolution family $\mathcal{U}(t,s)$ solves the NCP problem $\eqref{NCPE}$ 
and admits an Euler-type approximation, i.e.,
\begin{equation}
\big|\big|\mathcal{U}(T,0) - \prod_{k=0}^{M-1} e^{\Delta t (\mathcal{A} + \mathcal{B})(k\Delta t)}\big|\big|_X \le C(T)(\Delta t)^{\beta-\alpha},
\end{equation}
where $T = M\Delta t$, $\alpha,\beta$ are constants defined in assumptions \ref{parabolic} and \ref{continuous-appendix}. 

In the sequel, we  analyze the error between $\prod_{k=0}^{M-1} e^{\Delta t (\mathcal{A} + \mathcal{B})(k\Delta t)}$ and
$\prod_{k=0}^{M-1} e^{\Delta t\mathcal{A}(k\Delta t)}e^{\Delta t\mathcal{B}(k\Delta t)}$. 
 
First, we list all the assumptions as follows:
\begin{asm}\label{analytical}
	\begin{enumerate}
		\item ${\mathcal{A}(t)}_{t\ge0}$ and ${\mathcal{B}(t)}_{t\ge0}$ are all linear operators (may be unbounded) on X,
		\item $\mathcal{D}(\mathcal{A}(t))$ are the same for all $t$ and dense in $X$, 
		\item $||\mathcal{B}(t)||_X < C$ for all $t \ge 0$,
		\item $\mathcal{A}(t)$ satisfies \ref{parabolic} and $\mathcal{A}(t) + \mathcal{B}(t)$ satisfies \ref{parabolic} and \ref{continuous-appendix},
		\item $||e^{\tau \mathcal{A}(t)}||_X \le 1$,$||e^{\tau \mathcal{B}(t)}||_X \le 1$,$||e^{\tau (\mathcal{A}(t)+\mathcal{B}(t))}||_X \le 1$ for all $\tau\ge0$.
	\end{enumerate}
\end{asm}
To obtain a convergence theorem, we need an extra assumption in $\mathcal{A}$ and $\mathcal{B}$.
\begin{asm}\label{commutator}
	For the commutator $[\mathcal{A}(t),\mathcal{B}(t)] = \mathcal{A}(t)\mathcal{B}(t) - \mathcal{B}(t)\mathcal{A}(t)$, we assume that there is a non-negative $\gamma$ with
	\begin{equation}
	\big|\big|[\mathcal{A}(t),\mathcal{B}(t)] x\big|\big|_X \le c_1 \big|\big| \mathcal{A}(t) x\big|\big|^{\gamma}_X ||x||_X^{1-\gamma}, ~\forall~ x \in \mathcal{D}(\mathcal{A}).
	\end{equation}
\end{asm}

Next is a standard result from \cite{jahnke2000error}, and we prove it here.

\begin{thm}\label{onestep}
	Suppose assumptions \ref{analytical} and \ref{commutator} are satisfied. We have the following error estimate for the operator splitting method, 
	\begin{equation}
	\big|\big|(e^{\tau \mathcal{A}(t)}e^{\tau \mathcal{B}(t)} -  e^{\tau(\mathcal{A}(t)+\mathcal{B}(t))})x\big|\big|_X \le 
	C_1 \tau^{2-\gamma} ||x||_X, ~\forall x \in X,
	\label{OPerrorestimate1}
	\end{equation}
	where $C_1$ depends only on $c_1$, $\gamma$ and $||\mathcal{B}||_X$.
\end{thm}
   \begin{proof}
	We use the freezing coefficient formula and obtain
	\begin{equation}
	e^{\tau(\mathcal{A}(t)+\mathcal{B}(t))}x = e^{\tau(\mathcal{A}(t))}x + \int_{0}^{\tau} e^{s\mathcal{A}(t)} \mathcal{B}(t) e^{(\tau-s)(\mathcal{A}(t)+\mathcal{B}(t))} x ds.
	\label{OPerrorestimate2}
	\end{equation} 
	Expressing the term $e^{(\tau-s)(\mathcal{A}(t)+\mathcal{B}(t))}$ using the integral form \eqref{OPerrorestimate2}, we have
	\begin{equation}
	e^{\tau(\mathcal{A}(t)+\mathcal{B}(t))}x = e^{\tau(\mathcal{A}(t))}x+ \int_{0}^{\tau} e^{s\mathcal{A}(t)} \mathcal{B}(t) e^{(\tau-s)\mathcal{A}(t)}x ds + R_1x,
	\label{OPerrorestimate3}
	\end{equation} 
	where 
	\begin{equation}
	R_1 = \int_{0}^{\tau} e^{s\mathcal{A}(t)}\mathcal{B}(t) \int_{0}^{\tau - s} e^{\sigma \mathcal{A}(t)} \mathcal{B}(t) e^{(\tau-s-\sigma)(\mathcal{A}(t)+\mathcal{B}(t))} d\sigma ds.
	\label{OPerrorestimate4}
	\end{equation} 
	We can easily verify that the term $R_1$ is bounded, i.e., $||R_1||_X \le \frac{1}{2} \tau^2 ||\mathcal{B}(t)||_X^2$.
	
	On the other hand side, we express the term $e^{\tau \mathcal{B}(t)}$ into exponential series and obtain 
	\begin{equation}
	e^{\tau \mathcal{A}(t)} e^{\tau \mathcal{B}(t)}x = e^{\tau \mathcal{A}(t)}x+ \tau e^{\tau \mathcal{A}(t)} \mathcal{B}(t)x + R_2 x,
	\label{OPerrorestimate5}
	\end{equation} 
	where $||R_2 ||_X \le \frac{1}{2} \tau^2 ||\mathcal{B}(t)||_X^2$.
	
	Denoted by $f(s) = e^{s\mathcal{A}(t)} \mathcal{B}(t) e^{(\tau - s)\mathcal{A}(t)} x$, we have
	\begin{equation}
	e^{\tau \mathcal{A}(t)} e^{\tau \mathcal{B}(t)}x - e^{\tau(\mathcal{A}(t)+\mathcal{B}(t))}x = \tau f(\tau) - \int_0^\tau f(s) ds + r = d + r,
	\label{OPerrorestimate6}
	\end{equation} 
	where $d = \tau f(\tau) - \int_0^\tau f(s) ds = \tau^2 \int_0^1 \theta f'(\theta \tau) d\theta$ and $ r = R_2x - R_1x$. 
	
	Since $f'(s) = e^{s \mathcal{A}(t)} [\mathcal{A}(t),\mathcal{B}(t)] e^{(\tau - s)\mathcal{A}(t)} x$, assumption \ref{commutator} implies  
	\begin{align}
	 \big|\big|e^{s \mathcal{A}(t)}[\mathcal{A}(t),\mathcal{B}(t)] e^{(\tau - s)\mathcal{A}(t)} x \big|\big|_X \le c_1|| e^{s \mathcal{A}(t)}||_X||\mathcal{A}(t) e^{(\tau - s)\mathcal{A}(t)} x \big|\big|^{\gamma}_X \big|\big| e^{(\tau - s)\mathcal{A}(t)} x \big|\big|_X^{1-\gamma}.  
	\label{OPerrorestimate8}
	\end{align}
	By using the property of analytic semigroup \ref{Analytic}, we know that 
	\begin{equation}
	\big|\big|\mathcal{A}(t) e^{(\tau - s)\mathcal{A}(t)} x \big|\big|_X \le C(\tau - s )^{-1} ||x||_X.
	\label{OPerrorestimate9}
	\end{equation}
	 Thus, we have
	\begin{align}
	 ||d||_X = \big|\big|\tau^2 \int_0^1 \theta f'(\theta \tau) d\theta\big|\big|_X 
	\le  \big|\tau^2 \int_0^1 C\theta (\tau - \theta \tau)^{-\gamma} d\theta\big| ||v||_X  
	= \frac{C}{(1-\gamma)(2-\gamma)}\tau^{2-\gamma} ||v||_X.
	\label{OPerrorestimate10}
	\end{align}
	Notice that $||r||_X \le \tau^2 ||\mathcal{B}||_X^2$. We finish the proof.
\end{proof}


Using the one step estimate obtained in Theorem \ref{onestep}, we finally obtain the error estimate
for the operator splitting method.
 
\begin{thm} \label{splitting}
	Suppose assumptions \ref{analytical} and \ref{commutator} hold true. We have the following 
	error estimate for the operator splitting method in solving the NCP \eqref{NCPE}.
	\begin{equation}
	\big|\big|\prod_{k=0}^{N-1} e^{\Delta t (\mathcal{A} + \mathcal{B})(k\Delta t)} - \prod_{k=0}^{N-1} e^{\Delta t\mathcal{A}(k\Delta t)}e^{\Delta t\mathcal{B}(k\Delta t)}\big|\big|_X \le C_1 (\Delta t)^{1-\gamma},
	\label{OPerror1}
	\end{equation} 
	where $C_1$ is a constant independent of $\gamma$.
\end{thm}
\begin{proof}
	We take $t = j\Delta t$ and $ s = (j-1)\Delta t $ for $j = 1,\cdots,M-1$ in Theorem \ref{onestep}, and by using the telescoping sum argument, we obtain that for any $x \in X$,
	\begin{align}
	&\big|\big|\prod_{k=0}^{M-1} e^{\Delta t (\mathcal{A} + \mathcal{B})(k\Delta t)}x - \prod_{k=0}^{M-1} e^{\Delta t\mathcal{A}(k\Delta t)}e^{\Delta t\mathcal{B}(k\Delta t)}x\big|\big|_X  \nonumber \\
	= &\Big|\Big|\sum_{j = 0}^{M-1} \prod_{k=j+1}^{M-1} e^{\Delta t (\mathcal{A} + \mathcal{B})(k\Delta t)} \big(e^{\Delta t (\mathcal{A} + \mathcal{B})(j\Delta t)} - e^{\Delta t\mathcal{A}(j\Delta t)}e^{\Delta t\mathcal{B}(j\Delta t)}\big)\prod_{l=0}^{j-1} e^{\Delta t\mathcal{A}(l\Delta t)}e^{\Delta t\mathcal{B}(l\Delta t)}x\Big|\Big|_X \nonumber\\
	\le& \sum_{j = 0}^{M-1} C_1 (\Delta t)^{2-\gamma} \big|\big|\prod_{l=0}^{j-1} e^{\Delta t\mathcal{A}(l\Delta t)}e^{\Delta t\mathcal{B}(l\Delta t)}x\big|\big|_X   
	\le  \sum_{j = 0}^{M-1} C_1 (\Delta t)^{2-\gamma} ||x||_X = C_1 (\Delta t)^{1-\gamma} ||x||_X.
	\label{OPerror2}
	\end{align}
	Thus, we finish the proof.
\end{proof}

\bibliographystyle{siam}
\bibliography{ZWpaperGRF20}

\end{document}